\documentclass[11pt,reqno]{amsart}
\usepackage[a4paper,margin=2.85cm]{geometry}
\usepackage[T1]{fontenc}
\usepackage[utf8]{inputenc}
\usepackage{lmodern,amsmath,amssymb,amsthm,mathtools}
\usepackage{enumitem,microtype}
\usepackage[hidelinks]{hyperref}
\usepackage[nameinlink,capitalize,noabbrev]{cleveref}
\numberwithin{equation}{section}
\newtheorem{theorem}{Theorem}[section]
\newtheorem{proposition}[theorem]{Proposition}
\newtheorem{lemma}[theorem]{Lemma}
\newtheorem{corollary}[theorem]{Corollary}
\theoremstyle{definition}
\newtheorem{definition}[theorem]{Definition}
\newtheorem{example}[theorem]{Example}
\theoremstyle{remark}
\newtheorem{remark}[theorem]{Remark}
\newcommand{\C}{\mathbb C}
\newcommand{\R}{\mathbb R}
\newcommand{\Q}{\mathbb Q}
\newcommand{\Z}{\mathbb Z}
\newcommand{\B}{\mathbb B}

\newcommand{\Jac}{J_{\!\R}}
\newcommand{\Repart}{\operatorname{Re}}
\newcommand{\Impart}{\operatorname{Im}}
\newcommand{\Crit}{\operatorname{Crit}_{\R}}

\hypersetup{pdftitle={Local and global invertibility of pluriharmonic mappings},
pdfauthor={Franc Forstneri\v{c}, David Kalaj, and Deguang Zhong}}

\title[Invertibility of pluriharmonic mappings]{Local and global invertibility of pluriharmonic mappings}

\author[D. Kalaj]{David Kalaj}
\address{David Kalaj, University of Montenegro, Faculty of Natural Sciences and
Mathematics, 81000 Podgorica, Montenegro}
\email{davidk@ucg.ac.me}

\author[D. Zhong]{Deguang Zhong}
\address{Deguang Zhong, Institute of Applied Mathematics, Shenzhen Polytechnic
University, Shenzhen 518055, Guangdong, P. R. China}
\email{zhongdg1014@szpu.edu.cn}


\begin{document}
\begin{abstract}
We prove a higher-dimensional Kneser theorem for pluriharmonic
maps on bounded Stein domains whose closures are topological
manifolds with boundary. A homeomorphism on the boundary and
containment in the target domain imply a homeomorphism of the
closures and a real-analytic diffeomorphism of the interiors.
We classify pluriharmonic homeomorphisms between round shells in $\C^n$
for $n\ge2$. Such maps exist exactly when the two inner-to-outer radius
ratios agree, and each is the restriction of a centered real similarity.
We also prove the local Lewy theorem in every complex dimension:
local injectivity is equivalent to a nonvanishing real Jacobian.
For smooth boundary data, a formula counts the critical points of
each real scalar projection with their holomorphic multiplicities.
The boundary criterion gives removability for maps on punctured domains.
For bounded pluriharmonic maps of controlled distortion, we obtain
Lipschitz estimates on bounded $C^2$ domains in complex dimension
at least two. For quasiconformal pluriharmonic homeomorphisms,
we derive bi-Lipschitz bounds for convex targets and for pairs of
domains with Dini-smooth boundaries under the corresponding
source regularity assumptions.
\end{abstract}
\maketitle

\section{Introduction}

When does a boundary homeomorphism determine an invertible
extension? For harmonic maps of planar domains, the
Rad\'o--Kneser--Choquet theory gives a classical answer.
A bounded Jordan domain is the region enclosed by a simple closed
curve in the plane. Kneser's conditional form says that the harmonic
extension of a homeomorphism onto the boundary of such a domain $D$
is injective precisely when its interior image is $D$.
For convex targets the maximum principle supplies image
containment \cite{Rado1926,Kneser1926,Choquet1945,Duren2004}.
In higher real dimensions, harmonic extensions can fold despite
homeomorphic boundary values \cite{Laugesen1996}.

For the complex unit ball, we prove that a pluriharmonic extension
of a boundary homeomorphism onto the boundary of a bounded domain
is injective if and only if its interior image lies in that domain.
This answers the question of Bshouty and Hengartner
\cite[Open Problem~2.4, p.~25]{BshoutyHengartner1994}. In particular, a
pluriharmonic extension of a self-homeomorphism of the unit sphere
is a real-analytic diffeomorphism of the open ball onto itself.
This answers Laugesen's question \cite{Laugesen1996}, also recorded
as Hengartner's Problem~3.10(b) \cite{BshoutyLyzzaik2010}.
The global theorem applies more generally to bounded Stein domains
whose closures are topological manifolds with boundary.

Throughout, a domain is open and connected, and $\C^n$ is
identified with $\R^{2n}$. The symbols $|\cdot|$ and
$\langle\cdot,\cdot\rangle$ denote the Euclidean norm and inner
product; $\Repart$ and $\Impart$ denote real and imaginary parts.
We write $\B^n=\{z\in\C^n:|z|<1\}$ for the complex unit ball,
and $B_r^m,S_r^{m-1}$ for the Euclidean ball and sphere of radius
$r>0$ centered at $0$, with the subscript omitted for unit spheres.
The vectors $e_j$ are the standard coordinate vectors. In coordinates
$z_j=x_j+iy_j$, we use
the Wirtinger derivatives
\[
 u_{z_j}=\tfrac12(u_{x_j}-iu_{y_j}),\qquad
 u_{\bar z_j}=\tfrac12(u_{x_j}+iu_{y_j}).
\]
A real-valued $C^2$ function $u$ is pluriharmonic if
$u_{z_j\bar z_k}=0$ for all $j,k$; locally it is the real part of
a holomorphic function. A complex-valued pluriharmonic function
is locally $h+\bar g$, with $h,g$ holomorphic
\cite{GunningRossi1965,Hormander1990}.
Pluriharmonicity of maps is understood componentwise. We write
$du=\partial u+\bar\partial u$, where
$\partial u=\sum_j u_{z_j}\,dz_j$ and
$\bar\partial u=\sum_j u_{\bar z_j}\,d\bar z_j$.
A map is real analytic if its real coordinate functions are locally
given by convergent real power series. A real-analytic diffeomorphism
is a bijection for which both the map and its inverse are real analytic.

\subsection*{Global invertibility and rigidity of round shells}

The sublevel set of a real function $v$ at height $c$ is
$\{z:v(z)\le c\}$. An exhaustion of a domain $\Omega$ is a
continuous real function whose sublevel sets are compact in
$\Omega$. For a real $C^2$ function $\rho$ on a domain in $\C^n$,
its Levi form at $z$ in direction $\xi\in\C^n$ is
\[
 \mathcal L_\rho(z;\xi)=
 \sum_{j,k=1}^n\rho_{z_j\bar z_k}(z)\xi_j\overline{\xi_k}.
\]
The function is plurisubharmonic if this form is nonnegative,
and strictly plurisubharmonic if it is positive for every
$\xi\ne0$. More generally, a continuous function is
plurisubharmonic if its restrictions to complex lines are
subharmonic. A domain in $\C^n$ is Stein precisely when it admits
a smooth nonnegative strictly plurisubharmonic exhaustion
\cite{Hormander1990}.
A topological $d$-manifold with boundary is locally modeled on
$\R^d$ or the closed half-space $\{x\in\R^d:x_d\ge0\}$.
Its manifold boundary consists of the points represented by
$x_d=0$ in boundary charts. For a domain closure, we require this
intrinsic manifold boundary to coincide with its usual topological
boundary.

\begin{theorem}[Global boundary theorem]\label{thm2}
Let $\Omega\subset\C^n$ be a bounded Stein domain, and suppose that
$\overline\Omega$ is a compact topological $2n$-manifold with boundary
$\partial\Omega$. Let $D\subset\R^{2n}$ be any bounded domain. Suppose that
$F:\overline\Omega\to\R^{2n}$ is continuous, pluriharmonic on $\Omega$,
and satisfies
\[
 F(\Omega)\subset D,\qquad
 F|_{\partial\Omega}:\partial\Omega\longrightarrow\partial D
 \quad\text{is a homeomorphism}.
\]
Then $F$ is a homeomorphism from $\overline\Omega$ onto $\overline D$.
Its restriction to $\Omega$ is a real-analytic diffeomorphism onto $D$.
\end{theorem}

For $0<r<R$, write $A(r,R)=\{z\in\C^n:r<|z|<R\}$ for the round
shell. In the plane, the Nitsche theorem gives the existence threshold
$R_*/r_*\ge (R/r+r/R)/2$ for harmonic homeomorphisms between annuli
\cite{IwaniecKovalevOnninen2011}. In higher complex dimensions,
pluriharmonicity determines both the radius ratio and the form of the map.

\begin{theorem}[Classification of round shells]\label{thm:shell-rationality}
Let $n\ge2$, $0<r<R$ and $0<r_*<R_*$. A pluriharmonic homeomorphism
$F:A(r,R)\to A(r_*,R_*)$ exists if and only if
\[
                       \frac{r_*}{R_*}=\frac rR .
\]
Every such map has the form $F(z)=(R_*/R)Uz$ for some real orthogonal
transformation $U\in O(2n)$ acting on $\C^n=\R^{2n}$.
\end{theorem}

Section~\ref{sec:spherical-shells} extends the map across the inner ball.
Appendix~\ref{app:round-shells} uses the two spherical boundary
conditions to classify the extension.

\smallskip
\noindent\emph{Sketch of the proof of Theorem~\ref{thm2}.}
The image condition and the boundary values make $F:\Omega\to D$
proper. The boundary homeomorphism gives $F$ degree one modulo two,
so every point of $D$ has a preimage. To prove uniqueness, we compare
the sections of the target by half-spaces with their preimages in
the source.

For a nonzero real linear functional $\ell$ on the target, the
function $\ell\circ F$ is pluriharmonic. Adding a small multiple of
a strictly plurisubharmonic exhaustion and a small real linear
perturbation lets us approximate the source sections by compact
sublevel sets whose critical points are nondegenerate. At each such
point the Morse index, the number of negative eigenvalues of the
real Hessian, is at most $n$. Indeed, a negative subspace of real
dimension greater than $n$ contains a complex line, on which strict
plurisubharmonicity gives a positive sum of the two real second
derivatives. As the approximating sublevels grow, only cells of
dimension at most $n$ are attached. Consequently, a cycle of
dimension at least $n$ in an open source section that bounds in
$\Omega$ already bounds in that section. Section~\ref{sec:morse}
establishes this estimate. The degree-one map then yields a comparison
in high degrees. Duality and the complementary half-space give the
remaining degrees at regular cuts. \v{C}ech cohomology lets us pass
to the limit as regular cuts decrease to an arbitrary closed cut.

Fix $y\in D$ and let $A=F^{-1}(y)$ be the set of its preimages.
Pair each $x\in\overline\Omega$ with all unit vectors $e$ such that
$\langle F(x)-y,e\rangle\le0$. For fixed $e$ this gives a closed
half-space section. For fixed $x$ the allowed directions form a
closed hemisphere if $F(x)\ne y$ and a full sphere if $F(x)=y$.
We fill each full sphere with a ball, so that the projection of
the enlarged space onto $\overline\Omega$ has contractible fibers.
The relative cohomology of the enlarged space with respect to the
space of pairs records the cohomology of $A$ shifted by $2n$, the
dimension of the balls that were attached. The same construction
for the identity map of $\overline D$ has the one-point fiber
$\{y\}$. The half-space comparisons identify the cohomology of
the two spaces of pairs, and the comparison on the closures
identifies the cohomology of the enlarged spaces. The resulting
comparison of relative groups shows that $A$ has the cohomology
of a point with coefficients modulo two.

A compact real-analytic set of positive dimension has nonzero
top-dimensional homology modulo two. Hence $A$ is finite, and
its degree-zero cohomology shows that it has exactly one point.
Thus $F$ is injective on $\Omega$. Together with the boundary
data this gives a homeomorphism of the closures. The local Lewy
theorem gives the real-analytic interior inverse. The detailed
comparison is in Appendix~\ref{app:global-topology}.

An embedded hypersurface is locally flat at a point if a
homeomorphism of ambient neighborhoods takes it to a hyperplane.
Brown's theorem gives the source closure a collar, a neighborhood
of $\partial\Omega$ identified with $\partial\Omega\times[0,1)$
and with boundary at height zero \cite{Brown1962}.

\begin{corollary}[Hengartner's boundary problem]\label{cor1}
Let $F:\overline{\B^n}\to\C^n$ be continuous and pluriharmonic in $\B^n$.
If its boundary restriction is a self-homeomorphism of
$\partial\B^n$, then $F$ is a self-homeomorphism of $\overline{\B^n}$
and a real-analytic diffeomorphism of $\B^n$.
\end{corollary}

\begin{corollary}[The Kneser equivalence for the complex ball]\label{cor2}
Let $D$ be a bounded domain whose boundary is an embedded topological
$(2n-1)$-sphere in $\R^{2n}$. Let $F:\overline{\B^n}\to\R^{2n}$
be continuous and pluriharmonic in $\B^n$, with boundary restriction
a homeomorphism onto $\partial D$. Then the following are equivalent.
\begin{enumerate}[label=\textup{(\roman*)}]
\item $F$ is injective in $\B^n$.
\item $F(\B^n)\subset D$.
\item $F(\B^n)=D$.
\end{enumerate}
\end{corollary}

Related results include the multiply connected planar theorem
of Duren and Hengartner \cite{DurenHengartner1997}, other planar
injectivity criteria \cite{Lyzzaik1998,AlessandriniNesi2009},
nonlinear and surface analogues
\cite{IwaniecKoskiOnninen2016,IwaniecOnninen2019,AdamowiczEtAl2020},
Sobolev extensions \cite{KoskiOnninen2021}, and differential
criteria on the complex ball \cite{HamadaKohr2015}.

For convex targets the maximum principle supplies image containment,
and the map yields a strictly plurisubharmonic exhaustion of the source.

\begin{corollary}\label{cor:convex-no-stein}
Let $\Omega\subset\C^n$ be a bounded domain, with $\overline\Omega$
a compact topological manifold with manifold boundary $\partial\Omega$.
Let $D\subset\R^{2n}$ be a bounded convex domain. Every continuous
$F:\overline\Omega\to\R^{2n}$ that is pluriharmonic in $\Omega$
and maps $\partial\Omega$ homeomorphically onto $\partial D$
is a homeomorphism of the closures and a real-analytic diffeomorphism
of the interiors.
\end{corollary}

\subsection*{Local nondegeneracy and its topology}

For a differentiable map $F$ between equal-dimensional real spaces,
$DF$ is its real derivative matrix and $\Jac F=\det DF$ its real
Jacobian determinant.
Lewy's planar theorem identifies local injectivity of a harmonic
map with nonvanishing Jacobian \cite{Lewy1936}; the corresponding
statement fails for general harmonic maps in higher dimensions
\cite{Wood1991}. The pluriharmonic question appears in
\cite[Open Problem~1.1]{BshoutyHengartner1994} and
\cite[Problem~3.10(a)]{BshoutyLyzzaik2010}.

Naser considered the four-dimensional assertion
\cite{Naser1973}. This paper builds on Kalaj's proof for $\C^2$,
which addresses the local-topological issue in Naser's argument,
and the ball-source boundary estimates in \cite{Kalaj2026Lewy}, and on the
all-dimensional monodromy proof of Zhong and Wang
\cite{ZhongWang2026Lewy}.

\begin{theorem}[Pluriharmonic Lewy theorem]\label{thm1}
Let $\Omega\subset\C^n$ be a domain and let $F:\Omega\to\C^n$
be pluriharmonic. For $p\in\Omega$, the following conditions are equivalent.
\begin{enumerate}[label=\textup{(\roman*)}]
\item $F$ is injective on a neighborhood of $p$.
\item $\Jac F(p)\ne0$.
\item $F$ is a local real-analytic diffeomorphism at $p$.
\end{enumerate}
\end{theorem}

\smallskip
\noindent\emph{Sketch of the proof.}
Suppose that $F$ is locally injective at $p$ and that $DF(p)$
is singular. Choose a nonzero real linear functional $\ell$ that
vanishes on the image of $DF(p)$. Then
\[
 u(x)=\ell(F(x))-\ell(F(p))
\]
has zero differential at $p$. By invariance of domain, $F$ maps a
neighborhood of $p$ homeomorphically onto an open neighborhood of
$F(p)$. Under this homeomorphism, the zero level of $u$ becomes a
piece of a hyperplane. By pluriharmonicity, $u=\Repart q$ near $p$
for a holomorphic function $q$ with $q(p)=0$ and $dq(p)=0$.

Choose a small closed ball about $p$. The part of a nearby smooth
level of $q$ in this ball is the Milnor fiber. The Milnor fibration
gives this fiber the same rational homology as the positive side
$\{\Repart q>0\}$ in small Euclidean balls about $p$. Moreover,
inclusions between these positive sides induce isomorphisms on
homology for sufficiently small nested balls. Choose a neighborhood
in which $F$ takes the zero level to a hyperplane section of a ball
and the positive side to an open half-ball. This neighborhood can be
placed between two small Euclidean balls about $p$. The inclusion of
the smaller positive side into the larger one then factors through
the contractible positive side of this neighborhood. The inclusion
induces an isomorphism, so the Milnor fiber has the rational homology
of a point.

Moving the nearby level once around zero gives a return map of the
Milnor fiber, called the monodromy. Its Lefschetz number is the
alternating sum of the traces on rational homology. Only degree zero
contributes here, so this number is $1$. A'Campo's theorem gives $0$
because $dq(p)=0$. This contradiction proves that the real Jacobian
is nonzero. The converse and the real analyticity of the local inverse
follow from the inverse function theorem. See Section~\ref{sec:reduction}
for the proof.

\subsection*{Critical points, removability, and metric consequences}

A scalar projection of $F$ is $\ell\circ F$ for a real linear
functional $\ell$.

For smooth boundary data, Theorem~\ref{thm:euler-critical-count}
counts the critical points of every scalar projection with their
holomorphic multiplicities. The formula allows both signs of the
normal derivative on different boundary components.
Its ball case supplies an invertibility test from the boundary
Jacobian, leading to removability of compact holes by Hartogs
extension.

For a real linear map $A$, write
$\|A\|=\max_{|v|=1}|Av|$ and
$\ell(A)=\min_{|v|=1}|Av|$ for its operator norm and minimum stretch.
For a set $E$, let $\operatorname{diam}E=\sup_{x,y\in E}|x-y|$.
A map $f$ is Lipschitz if $|f(x)-f(y)|\le L|x-y|$ for some
$L>0$, co-Lipschitz if $|f(x)-f(y)|\ge c|x-y|$ for some $c>0$,
and bi-Lipschitz if both inequalities hold. Its Lipschitz constant is
$\operatorname{Lip}(f)=\sup_{x\ne y}|f(x)-f(y)|/|x-y|$.
For $C^1$ homeomorphisms between domains in $\R^d$,
quasiconformality means that $\|Df\|^d\le K|\Jac f|$ for some
finite $K\ge1$; the full Sobolev convention is given in
Section~\ref{sec:metric-notation}.

A disk estimate in a direction complex orthogonal to the radius
gives Lipschitz bounds on the ball
(Proposition~\ref{prop:qc-lipschitz-ball}). Applying it on interior
tangent balls yields the bound for $C^2$ sources and bi-Lipschitz
continuity for bounded convex targets
(Corollary~\ref{cor:qc-smooth-source}). Theorem~\ref{thm:pavlovic}
gives a boundary criterion for bi-Lipschitz continuity.
Appendix~\ref{sec:boundary} extends the Dini boundary estimates of
\cite{Kalaj2026Lewy} from balls to general sources with the boundary
regularity specified in Definition~\ref{def:dini}.

Positive constants $c,C$ may vary, with their dependence specified
when needed. For nonnegative quantities, $a\asymp b$ means
$cb\le a\le Cb$ with uniform positive constants. For $b>0$ near a
stated limit, $a=O(b)$ means $|a|\le Cb$, and $a=o(b)$ means
$|a|/b\to0$.

\section{Half-space sections of the source}\label{sec:morse}

We write $\Z$ for the integers, $\Q$ for the rational
numbers, and $\mathbb F_2=\Z/2\Z$ for the field with
two elements. For spaces $A\subset Y$ and a coefficient group $G$,
write $H_j(Y;G)$ and $H_j(Y,A;G)$ for ordinary and relative
singular homology in degree $j$. A singular $j$-simplex in $Y$
is a continuous map from the standard simplex
$\Delta^j=\{(t_0,\ldots,t_j)\in[0,1]^{j+1}:\sum_i t_i=1\}$ to $Y$.
A $j$-chain is a finite formal sum of such maps with coefficients in
$G$. A cycle is a chain with zero boundary. Its homology class is
zero exactly when it is the boundary of a $(j+1)$-chain.
In relative homology, chains in $A$ are set equal to zero.
A tilde denotes reduced groups, which remove the constant summand
in degree zero for nonempty spaces and agree with the ordinary
groups in positive degrees. For a continuous map $g$,
$g_*$ and $g^*$ denote the induced maps in homology and
cohomology, respectively.
A nonempty space is acyclic over $G$ if its reduced groups vanish
in the indicated homology or cohomology theory.

For a smooth real function $v$, $D^2v$ denotes its real Hessian,
the matrix of second partial derivatives in real coordinates.
We write $D^2v(z)[\xi,\xi]$ for the second derivative of
$s\mapsto v(z+s\xi)$ at $s=0$, with $s$ real.
A critical point is a point where $dv=0$, and it is nondegenerate
if the Hessian there is nonsingular. A Morse function has only
nondegenerate critical points. Its Morse index at a critical point is
the number of negative Hessian eigenvalues counted with multiplicity.
For example, $v(x,y)=x^2-y^2$ has index one at the origin. A real
number $a$ is a regular value of $v$ if $dv\ne0$ at every point of
the level $\{v=a\}$.

We write $X\simeq Y$ when there are continuous maps between $X$
and $Y$ whose compositions are homotopic to the identity maps;
a homotopy is a continuous family of maps parametrized by $[0,1]$.
A space is contractible if its identity is homotopic to a constant
map.
The homology groups in the following lemma have coefficients in
$\mathbb F_2$.

\begin{lemma}\label{lem4}
Suppose that $\Omega$ and $F$ satisfy the hypotheses of
Theorem~\ref{thm2}. For a nonzero real linear functional
$\ell:\R^{2n}\to\R$ and $t\in\R$, set
$\Omega_{\ell,t}=\{z\in\Omega:\ell(F(z))<t\}$.
For every $j\ge n$, the inclusion induces an injection
\[
 H_j(\Omega_{\ell,t};\mathbb F_2)\longrightarrow H_j(\Omega;\mathbb F_2).
\]
\end{lemma}

\begin{proof}
Put $u=\ell\circ F$. Choose a smooth nonnegative strictly
plurisubharmonic exhaustion $\rho$ of $\Omega$. For $\varepsilon>0$
and a real linear functional $\lambda$ on the source, the function
$v=u+\varepsilon\rho+\lambda$ is strictly plurisubharmonic
because both $u$ and $\lambda$ have zero Levi form. At a
nondegenerate critical point $z$, let $E$ be the real span of the
Hessian eigenvectors with negative eigenvalues. If $\dim_\R E>n$,
then $E\cap iE$ contains a nonzero vector $\xi$, so both $\xi$ and
$i\xi$ belong to $E$. The sum
\[
 D^2v(z)[\xi,\xi]+D^2v(z)[i\xi,i\xi]
   =4\mathcal L_v(z;\xi)
\]
would be negative. Strict plurisubharmonicity makes the right-hand
side positive, a contradiction. Thus the Morse index is at most
$n$. Arbitrarily small choices of $\lambda$ for which $v$ is Morse
exist by Sard's
theorem applied to the gradient of $u+\varepsilon\rho$.

Let a homology class in $\Omega_{\ell,t}$ be represented by a
cycle with nonempty compact support $C$ and suppose it bounds in
$\Omega$. Set $\gamma=(t-\max_C u)/4$. Choose $\varepsilon$ and
$\lambda$ so that $\varepsilon\max_C\rho<\gamma$ and
$\sup_\Omega|\lambda|<\gamma$. Then $v\le\max_Cu+2\gamma$ on
$C$ and $v\ge u-\gamma$ everywhere. Choose a regular value
$a$ between $\max_Cu+2\gamma$ and $t-\gamma$. The sublevel
$M=\{v\le a\}$ is compact and satisfies
$C\subset M\subset\Omega_{\ell,t}$. Compactness follows from
boundedness of $u,\lambda$ and the exhaustion property of $\rho$.
The bounding chain lies in a larger compact regular sublevel
$M'$ of $v$. Morse theory describes the change in a sublevel set
across a critical point of index $k$, up to homotopy, by attaching
a $k$-dimensional cell. This means gluing a closed $k$-ball to
the preceding space by a map defined on its boundary.
All cells added between $M$ and $M'$ have dimension at most $n$,
so $H_{j+1}(M',M;\mathbb F_2)=0$ for $j\ge n$
\cite{AndreottiFrankel1959,Milnor1963}. In the exact sequence
\[
 H_{j+1}(M',M;\mathbb F_2)\longrightarrow H_j(M;\mathbb F_2)
 \longrightarrow H_j(M';\mathbb F_2),
\]
the kernel of the second arrow is th e image of the first. The
first group is zero, so a cycle in $M$ that bounds in $M'$ already
bounds in $M$.
\end{proof}

\section{Consequences and examples}\label{sec:applications}

A map is proper if the inverse image of every compact set is compact.

\subsection{Convex and Jordan targets}

\begin{proof}[Proof of Corollary~\ref{cor1}]
More generally, suppose in the setting of Theorem~\ref{thm2} that
$D$ is convex and that only the boundary hypothesis has been
imposed. For each nonzero real linear functional $\ell$, the
continuous scalar function $\ell\circ F$ is harmonic in $\Omega$.
It is nonconstant because its boundary values range over
$\ell(\partial D)$. The strong maximum principle gives
$\ell\circ F<\max_{\overline D}\ell$ in $\Omega$. A supporting or
separating hyperplane exists at every point of $\R^{2n}\setminus D$.
Thus $F(\Omega)\subset D$. Theorem~\ref{thm2} applies, in particular
when $\Omega=D=\B^n$.
\end{proof}

\begin{proof}[Proof of Corollary~\ref{cor2}]
If $F$ is injective in the ball, invariance of domain gives
a homeomorphism of the open ball onto the open set $F(\B^n)$.
No point $b$ of the sphere can map into this open set. Otherwise
a sequence of interior points converging to $b$ would have
images converging to $F(b)$ and, by continuity of the inverse
on $F(\B^n)$, would converge to an interior point. Compactness
and continuity also give
$\overline{F(\B^n)}=F(\overline{\B^n})$. Hence the boundary of
$F(\B^n)$ is exactly $\partial D$. By Jordan--Brouwer separation,
$F(\B^n)$ and $D$ must both be the bounded complementary
component of $\partial D$. This proves (i)$\Rightarrow$(iii).
The implication (iii)$\Rightarrow$(ii) is immediate, and
Theorem~\ref{thm2} proves (ii)$\Rightarrow$(i).
\end{proof}

\begin{proof}[Proof of Corollary~\ref{cor:convex-no-stein}]
The supporting-hyperplane argument used for Corollary~\ref{cor1}
gives $F(\Omega)\subset D$, so the interior map is proper.
For any proper domain $D\subsetneq\R^d$, write
$\delta_D(y)=\inf_{a\in\partial D}|y-a|$ for the distance to its
boundary. For the bounded convex domain here, the support function
is $h_D(\xi)=\sup_{w\in D}\langle w,\xi\rangle$. For unit $\xi$,
the deficit $h_D(\xi)-\langle y,\xi\rangle$ is the distance from
$y\in D$ to the supporting hyperplane perpendicular to $\xi$.
The largest ball centered at $y$ and contained in $D$ therefore has
radius
\[
 \delta_D(y)=\inf_{|\xi|=1}\bigl(h_D(\xi)-\langle y,\xi\rangle\bigr).
\]
Thus $\delta_D$ is concave and $\psi=-\log\delta_D$ is convex.
Along each complex line, the scalar coordinates of $F$ obey the
harmonic mean-value formula. Jensen's inequality consequently
gives the sub-mean inequality for $\psi\circ F$ on every disk
in such a line. Hence $\psi\circ F$ is continuous and
plurisubharmonic. If a sequence leaves every compact subset of
$\Omega$, properness and boundedness of $D$ force its image
distance to $\partial D$ to tend to zero. Therefore
$|z|^2+\psi(F(z))$ is a continuous strictly plurisubharmonic
exhaustion. Here strictness means that subtracting a sufficiently
small multiple of $|z|^2$ leaves a plurisubharmonic function.
Richberg's approximation theorem gives a smooth
strictly plurisubharmonic exhaustion, using a bounded positive
approximation error \cite{Richberg1968}. This is the exhaustion
required in the proof of Theorem~\ref{thm2}.
\end{proof}

\subsection{The image condition for nonconvex targets}

\begin{example}[Escape from a nonconvex target]\label{ex1}
Write $z=(x+is,z')$ and set
\[
 F_a(z)=(x+i(s+a(x^2-s^2)),z'),\qquad a>0.
\]
This is a pluriharmonic polynomial on $\C^n$. On $\partial\B^n$
it agrees with the real-analytic diffeomorphism
\[
 G_a(z)=(x+i(s+a(2x^2+|z'|^2-1)),z')
\]
of the ambient real space. The inverse of $G_a$ just subtracts
the same translation in the $s$ coordinate. In particular
$F_a$ maps the sphere homeomorphically onto the boundary of the
real-analytic domain $D_a=G_a(\B^n)$.

When $a>1/2$, the scalar map $s\mapsto s-as^2$ at $x=0$, $z'=0$
folds inside the interval $(-1,1)$. At $s=1/(2a)$ its value is
$1/(4a)$, whereas the upper endpoint of the corresponding
section of $D_a$ is $1-a$. The difference is
$(2a-1)^2/(4a)>0$. Thus $F_a$ is noninjective and sends an
interior point outside $\overline{D_a}$.

For $1/4<a\le1/2$, the same family gives diffeomorphisms onto
nonconvex targets. For fixed $(x,z')$ the derivative in the
remaining coordinate is $1-2as>0$ in the open ball, and the
two boundary values are those of $G_a$. This proves
$F_a(\B^n)=D_a$. At $z'=0$, the upper boundary graph of $D_a$
has second derivative $4a-1>0$ at $x=0$, which excludes
convexity.

For $1/4<a<1/2$, the Jacobian is positive on $\overline{\B^n}$, so
Corollary~\ref{cor:ball-boundary-jacobian} also applies. At $a=1/2$,
the map is a diffeomorphism of the open domains and its boundary
Jacobian vanishes at $z_1=i$, $z'=0$.
\end{example}

\section{Local invertibility}\label{sec:reduction}

A holomorphic germ $q:(\C^n,0)\to(\C,0)$ is a holomorphic
function near $0$ with $q(0)=0$, considered up to agreement on
a smaller neighborhood. We use the same convention for germs
of sets. For a scalar function $u$, $\nabla u$ is its real gradient.
For a holomorphic $q$, $dq$ is its complex differential.

\begin{lemma}\label{lem:primitive}\label{lem:real-complex}
Every real-valued pluriharmonic function $u$ on a ball in $\C^n$
is $\Repart q$ for a holomorphic function $q$, unique up to an
imaginary constant. Moreover, $dq=2\partial u$ and
\begin{equation}\label{eq:gradient-identity}
 |\nabla u|^2=\sum_{j=1}^n|q_{z_j}|^2.
\end{equation}
In particular, $du(p)=0$ if and only if $dq(p)=0$.
\end{lemma}

\begin{proof}
The form $2\partial u$ is holomorphic and closed, so the
holomorphic Poincar\'e lemma gives a primitive $q$ on the ball.
Its real part differs from $u$ by a real constant, which can be
absorbed into $q$. The Cauchy--Riemann equations give
\eqref{eq:gradient-identity}; two primitives with the same real
part differ by an imaginary constant.
\end{proof}

\begin{lemma}\label{lem:flat-obstruction}
Let $q:(\C^n,0)\to(\C,0)$ be a nonconstant holomorphic germ.
If $\{\Repart q=0\}$ is locally flat at $0$, then $dq(0)\ne0$.
\end{lemma}

\begin{proof}
Suppose that $dq(0)=0$. For a sufficiently small $\varepsilon>0$
and a sufficiently small $\eta>0$ chosen after $\varepsilon$, the
compact Milnor fiber is
\[
 M_q=q^{-1}(\eta)\cap\overline B_\varepsilon^{2n}.
\]
Put $W_r=B_r^{2n}\cap\{\Repart q>0\}$. The map
$z\mapsto |z|q(z)/|q(z)|$ on $W_r$ is a fiber bundle over the
right half-disc of radius $r$, so locally it is a product
projection. It is a product over a larger half-disc because the
base is contractible, and that product restricts to the smaller
half-discs. Hence $W_r\hookrightarrow W_R$
is a homotopy equivalence for sufficiently small $0<r<R$.
The fiber over a positive real value is a fiber of $q/|q|$ on
a sphere. It has the homotopy type of $M_q$ by the sphere--tube
comparison \cite[Theorems~2--3 and Section~3.1]{CisnerosSeadeSnoussi2009}.

Choose a neighborhood $V$ of $0$ inside $B_R^{2n}$ that a
flattening homeomorphism takes to a ball, with the zero set
mapped to a real hyperplane section through its center.
The two complementary components have opposite signs because
the nonconstant harmonic function $\Repart q$ takes both signs
near $0$. Thus $V\cap\{\Repart q>0\}$ is contractible. For
$B_r^{2n}\subset V$, the inclusion $W_r\hookrightarrow W_R$
factors through this contractible side. The inclusion induces
an isomorphism on reduced rational homology, so those groups
vanish for $W_R$ and hence for $M_q$.

The tube fibration is the restriction of $q$ to
$\{z\in\overline B_\varepsilon^{2n}:|q(z)|=\eta\}$, with target the
circle of radius $\eta$. Its monodromy $h_q:M_q\to M_q$ is obtained
by transporting the fiber once around that circle. The compact Milnor fiber has finite-dimensional
rational homology vanishing in high degrees. The Lefschetz number of
$h_q$ is the alternating sum
\[
 \Lambda(h_q)=\sum_{j\ge0}(-1)^j
       \operatorname{tr}\bigl((h_q)_*\mid H_j(M_q;\Q)\bigr).
\]
Only $H_0(M_q;\Q)\cong\Q$ is nonzero, and $(h_q)_*$ is the
identity on that group. Thus $\Lambda(h_q)=1$. A'Campo's
theorem gives $\Lambda(h_q)=0$ when $dq(0)=0$
\cite[Th\'eor\`eme~1 bis]{ACampo1973}, a contradiction.
\end{proof}

\begin{proof}[Proof of Theorem~\ref{thm1}]
Suppose that $F$ is injective on a neighborhood of $p$ and that
$DF(p)$ is singular. Choose a nonzero real linear functional
$\ell$ vanishing on the image of $DF(p)$ and set
$u=\ell\circ F-\ell(F(p))$. By invariance of domain, $F$ is a
homeomorphism onto an open set near $F(p)$. Thus $u$ is
nonconstant. Lemma~\ref{lem:primitive} writes $u$ locally as
$\Repart q$ with $q(p)=0$ and $dq(p)=0$. The zero set of $u$
is the inverse image of the affine hyperplane
$\{\ell=\ell(F(p))\}$ and is locally flat. Lemma~\ref{lem:flat-obstruction},
applied after translating $p$ to the origin, gives a contradiction.
The real-analytic inverse function theorem gives the converse
and the regularity of the local inverse.
\end{proof}

A continuous map from a real $d$-dimensional domain to $\R^k$
is a topological submersion at a point if local homeomorphisms of
the source and target turn it into the projection onto the first
$k$ coordinates.

\begin{corollary}[Topological submersions]\label{cor:ph-submersion}
Let $F:U\subset\C^n\to\R^k$, $1\le k\le2n$, be pluriharmonic.
Then $F$ is locally equivalent to a coordinate projection by
homeomorphisms at $p\in U$ if and only if $DF(p)$ has rank $k$.
In that case the local coordinates can be chosen real analytic.
\end{corollary}

\begin{proof}
The analytic submersion theorem gives the local coordinates under
the rank condition. Conversely, a topological product representation
is locally open and pulls back any affine target hyperplane through
$F(p)$ to a locally flat hypersurface through $p$. Indeed, in the
product coordinates that hyperplane first becomes a locally flat
target hypersurface, and its product with the fiber coordinates
is locally flat. If the rank were less than $k$, a nonzero real
covector annihilating $DF(p)$ would yield a nonconstant real
pluriharmonic function critical at $p$ whose level is that
hypersurface. A holomorphic primitive of its complex differential,
normalized to vanish at $p$, contradicts Lemma~\ref{lem:flat-obstruction}.
\end{proof}

For a differentiable map $F:\Omega\subset\C^n\to\C^n$, define its
real critical set by
\[
 \Crit(F):=\{p\in\Omega:\Jac F(p)=0\}.
\]
For a continuous map $F$ on $\Omega$, its local noninjectivity set is
\[
 \mathcal B_{\mathrm{inj}}(F):=
 \{p\in\Omega:F\text{ is not injective on any neighborhood of }p\}.
\]
A real-analytic set is a subset locally defined by the simultaneous
vanishing of real-analytic functions.

\begin{corollary}
\label{cor:dichotomy}
Let $F:\Omega\subset\C^n\to\C^n$ be pluriharmonic on a connected
domain $\Omega$. Exactly one of the following occurs:
\begin{enumerate}[label=\textup{(\roman*)}]
\item $\Jac F\equiv0$, and $F$ is locally injective nowhere;
\item $\mathcal B_{\mathrm{inj}}(F)$ is a proper closed real-analytic subset
of $\Omega$ of real dimension at most $2n-1$, hence has empty interior and
$2n$-dimensional Lebesgue measure zero. On its open dense complement, $F$ is
a local real-analytic diffeomorphism.
\end{enumerate}
\end{corollary}

\begin{proof}
If $\Jac F\equiv0$, the first alternative follows from
Theorem~\ref{thm1}. Otherwise the zero set of the nontrivial real
analytic function $\Jac F$ is a proper real-analytic subset of codimension at
least one, and therefore has empty interior and measure zero; see, for example,
\cite{KrantzParks2002}. Apply Theorem~\ref{thm1} again.
\end{proof}

\begin{example}\label{ex:critical-hyperplane}
For every $n\ge1$, define $F:\C^n\to\C^n$ by
\[
 F(z_1,\ldots,z_n)
 =\bigl(\Repart z_1+i\Impart(z_1^2),z_2,\ldots,z_n\bigr).
\]
Define the holomorphic polynomials $h_1=(z_1+z_1^2)/2$ and
$g_1=(z_1-z_1^2)/2$. The first coordinate is
$h_1+\overline{g_1}$, so $F$ is pluriharmonic.
Writing $z_1=x+iy$, its first real block is $(x,2xy)$ and has determinant
$2x$. Thus $\Jac F=2\Repart z_1$ and the critical set is a real hyperplane.
At $x=0$ the entire local $y$-segment is collapsed, whereas at $x\ne0$ the
map is locally invertible. Hence the dimension $2n-1$ in
Corollary~\ref{cor:dichotomy} is attained.
\end{example}

\section{Euler characteristic and smooth boundary invertibility}
\label{sec:euler-index}

When the homology groups are finite dimensional
and vanish in high degrees, their alternating dimension sum
$\chi(Y)=\sum_{j\ge0}(-1)^j\dim_\Q H_j(Y;\Q)$ is the Euler
characteristic.

A domain has $C^k$ boundary if its boundary is locally a graph
of a $C^k$ function and the domain lies on one side. In
$C^1(\overline\Omega)$, the function and its first derivatives extend
continuously to the boundary.

For a map $g:M\to N$ between closed connected oriented
$d$-manifolds, its integer degree is defined by
$g_*[M]=(\deg g)[N]$ in $H_d(N;\Z)$, where $[M]$ and $[N]$
are the fundamental classes determined by the orientations.
Boundary orientations use the outward-normal-first convention.
At an isolated zero $p$ of a vector field $X$, its local index
$\operatorname{ind}_p(X)$ is the degree of $X/|X|$ on a small sphere
about $p$.
Write $\mathcal O_{\C^n,p}$ for the ring of holomorphic germs at $p$
and $(a_1,\ldots,a_n)$ for the ideal generated by these germs.
For a holomorphic map $a=(a_1,\ldots,a_n)$ with an isolated zero at
$p$, its complex multiplicity is
\[
 \operatorname{mult}_p(a)=
 \dim_\C\bigl(\mathcal O_{\C^n,p}/(a_1,\ldots,a_n)\bigr).
\]
For $2\partial u$ this is the multiplicity of its coefficient map
$(2u_{z_1},\ldots,2u_{z_n})$.

\begin{theorem}[Euler characteristic count]\label{thm:euler-critical-count}
Let $\Omega\subset\C^n$ and $D\subset\R^{2n}$ be bounded domains with
$C^2$ boundary. Let $F\in C^1(\overline\Omega,\R^{2n})$ be
pluriharmonic in $\Omega$, and suppose that
$\varphi=F|_{\partial\Omega}:\partial\Omega\to\partial D$ is a
$C^1$ diffeomorphism with $\Jac F\ne0$ on $\partial\Omega$.
Write $\nu_\Omega$ and $\nu_D$ for the outward unit normals.
For each connected component $C$ of $\partial\Omega$ define
\[
 \varepsilon_C=\operatorname{sgn}
    \langle DF(x)\nu_\Omega(x),\nu_D(\varphi(x))\rangle,
    \qquad x\in C.
\]
For $C'=\varphi(C)$, let $d(C')$ denote the degree of
$\nu_D|_{C'}:C'\to S^{2n-1}$, with $C'$ oriented as part of $\partial D$.
Put
\begin{equation}\label{eq:euler-defect}
 \mathcal E(F)=\chi(\Omega)-\chi(D)
                  +2\sum_{\varepsilon_C=-1}d(\varphi(C)).
\end{equation}
For every nonzero real linear functional $\lambda$, the critical set
of $u=\lambda\circ F$ is finite, and
\begin{equation}\label{eq:euler-critical-count}
 \sum_{p\in\Omega:\,du(p)=0}\operatorname{mult}_p(2\partial u)
    =(-1)^n\mathcal E(F).
\end{equation}
The following conditions
are equivalent: $\mathcal E(F)=0$; $\Jac F$ is nowhere zero in $\Omega$;
and $F$ is a $C^1$ diffeomorphism of $\overline\Omega$ onto
$\overline D$. In this case the interior inverse is real analytic.
\end{theorem}

\begin{proof}
The boundary differential is invertible and its tangential part equals
$D\varphi$, so the normal component defining $\varepsilon_C$ has
constant nonzero sign on each component. The holomorphic coefficient map
$a=(2u_{z_1},\ldots,2u_{z_n})$ extends continuously to $\overline\Omega$
and is nonzero at the boundary. Its zero set is compact in $\Omega$.
A compact complex analytic subset of a domain in $\C^n$ is finite by
the maximum principle on its irreducible components
\cite{GunningRossi1965}. The local Brouwer degree of $a$ at a zero
is its positive complex multiplicity. In complex notation the real
gradient $\nabla u$ equals $\bar a$, and conjugation has real
orientation sign $(-1)^n$. Hence
\begin{equation}\label{eq:gradient-complex-index}
   \operatorname{ind}_p(\nabla u)=(-1)^n\operatorname{mult}_p(2\partial u).
\end{equation}

Identify source and target outward collars by
\[
 \Psi(x+t\nu_\Omega(x))=\varphi(x)+t\nu_D(\varphi(x)).
\]
Let $v$ be the unit vector representing the direction of $\lambda$.
On $C$ define the orthogonal map $R_C$ to be the identity on the
boundary tangent plane and multiplication by $\varepsilon_C$ on the
normal line. The maps $DF$ and $D\Psi R_C$ have the same tangential
part and the same sign of the normal component. Their straight-line
interpolation is invertible. Transposition and evaluation at $v$
give a nonvanishing homotopy of the boundary gradient, up to a
positive constant, to $(D\Psi R_C)^Tv$. For any invertible $L$,
the relation $L^Tv=(L^TL)L^{-1}v$ and positivity of $L^TL$ give a
nonvanishing homotopy to $L^{-1}v$. With $L=D\Psi R_C$ this last
field is $D\Psi^{-1}W_C$, where, on $C'=\varphi(C)$,
\[
 W_C=\begin{cases}
       v,&\varepsilon_C=1,\\
       v-2\langle v,\nu_D\rangle\nu_D,&\varepsilon_C=-1 .
     \end{cases}
\]
Here $D\Psi$ sends tangent vectors to tangent vectors and the source
unit normal to the target unit normal, which justifies moving the
reflection to the target side.

Extend $W_C$ over an outward target collar of $C'$ to a field equal
to the outward normal at the outer surface. The sum of its interior
indices is $d(C')-\deg(W_C:C'\to S^{2n-1})$, by the oriented
boundary formula for degree. A perturbation in the open collar
allows the indices to be taken at isolated zeros. For
$\varepsilon_C=1$, the second degree is zero. For $\varepsilon_C=-1$,
it equals $2d(C')$: the map
\[
  R_v:S^{2n-1}\to S^{2n-1},
  \qquad R_v(\eta)=v-2\langle v,\eta\rangle\eta
\]
has degree $2$. The regular value $-v$ has preimages $v$ and $-v$;
the tangential derivatives there are $-2I$ and $2I$, respectively.
The oriented tangent planes at the antipodal points have opposite
orientations in their common underlying hyperplane. Each local
degree is thus $+1$.

Attach a collar realizing the nonvanishing gradient homotopy and
then a collar carrying the pullback of the corresponding target
field. Pullback by a diffeomorphism preserves local indices, and
the field at the outer edge points outward. The Poincar\'e--Hopf
formula \cite[Section~6]{Milnor1997} gives
\[
 \sum_{du(p)=0}\operatorname{ind}_p(\nabla u)
  =\chi(\Omega)-\sum_C
       \bigl(d(\varphi(C))-\deg(W_C)\bigr).
\]
Since $\sum_C d(\varphi(C))=\chi(D)$ by the same formula for the
outward normal on $\partial D$, this identity and
\eqref{eq:gradient-complex-index} prove \eqref{eq:euler-critical-count}.

If $\mathcal E(F)=0$, positivity of multiplicities gives
$d(\lambda\circ F)\ne0$ for every nonzero $\lambda$. Equivalently,
$DF$ is invertible. Conversely, an everywhere invertible $DF$
gives zero in \eqref{eq:euler-critical-count}. It remains to show
that a map with this interior and boundary nondegeneracy has image
$D$ and is injective.

Write $N(y)$ for the number of interior preimages of
$y\notin\partial D$. This is finite and constant on each component of
$\R^{2n}\setminus\partial D$, since over such a component the
interior map is a proper local diffeomorphism. Each connected
hypersurface component $C'$ of $\partial D$ separates the ambient
sphere into two components by Alexander duality. The side disjoint
from $D$ contains no other part of $\partial D$, since
$\partial D\subset\overline D$. It is the complementary chamber
adjacent to $C'$, with that component as its entire boundary. There
is exactly one $C'$ adjacent to the unbounded chamber.

Across $C'=\varphi(C)$, the preimages away from the source boundary
have the same local count on both sides, while precisely one
additional sheet comes from the source boundary collar. This sheet
is on the $D$ side for $\varepsilon_C=1$ and on the complementary
side for $\varepsilon_C=-1$. To justify the count, use local
invertibility at each preimage of a fixed boundary point and use
compactness of $\overline\Omega$ to exclude additional preimages of
nearby target points. The unbounded chamber has $N=0$ by boundedness,
so its adjacent component has $\varepsilon_C=1$ and $N(D)=1$.

Suppose that $\varepsilon_C=-1$ at another component. Its bounded
complementary chamber $E$ then has $N(E)=2$.
The double covering $F^{-1}(E)\to E$ splits over a collar of $C'$ in
$E$: one sheet is the image under the inverse furnished by the source
collar of $C$, and the other is the single remaining sheet.
Its monodromy, a homomorphism to $\mathbb Z/2\mathbb Z$, thus vanishes
on the image of $H_1(C';\mathbb F_2)$. The map
$H_1(C';\mathbb F_2)\to H_1(E;\mathbb F_2)$ is onto by the
Mayer--Vietoris sequence for $S^{2n}$ cut along $C'$, using collars
and $H_1(S^{2n};\mathbb F_2)=0$. The covering splits over all of $E$.
The component containing the collar sheet is open in $\Omega$.
It is also closed in $\Omega$: an accumulation point outside that
component would have image on $C'$, but near $C'$ that sheet lies
in the source collar and converges to the source boundary.
This contradicts connectedness of $\Omega$, since $N(D)=1$.
Consequently every $\varepsilon_C=1$ and every complementary
chamber has $N=0$. Local invertibility excludes interior points
mapping to $\partial D$. The interior and boundary maps therefore
give a bijection of the closures. The inverse function theorem,
using local $C^1$ extensions across the boundary, gives the claimed
regularity of the inverse.
\end{proof}

\begin{corollary}[Boundary Jacobian criterion for a ball]
\label{cor:ball-boundary-jacobian}
Let $B_s=\{z\in\C^n:|z|<s\}$, $s>0$, and let $D\subset\R^{2n}$
be a bounded $C^2$ domain. Suppose $F\in C^1(\overline B_s,\R^{2n})$
is pluriharmonic in $B_s$ and its trace is a $C^1$ diffeomorphism
from $\partial B_s$ onto $\partial D$. Then $F$ is a $C^1$
diffeomorphism of the closures if and only if $\Jac F$ is nonzero
on $\partial B_s$. In the affirmative case the interior inverse is
real analytic.
\end{corollary}

\begin{proof}
Only sufficiency requires proof. Since the boundary differential
is invertible, the normal component defining $\varepsilon_C$ is nonzero. It has constant sign
on the connected sphere. Choose a unit vector $e$ and a point
$y_0\in\partial D$ at which $\langle e,y\rangle$ attains its maximum
over $\overline D$. The outward normal there is $e$. At
$x_0=\varphi^{-1}(y_0)$, the harmonic function
$\langle e,F\rangle$ attains its boundary maximum. Its outward
derivative is nonnegative by the maximum principle and is nonzero
by boundary invertibility, since its tangential derivative at $x_0$
vanishes. The normal component is therefore positive everywhere.
The embedded sphere $\partial D$ separates $\R^{2n}$ into two
components. Alexander duality shows that the bounded component $D$
has trivial reduced homology \cite{Hatcher2002}. Thus $\chi(D)=1$
and Theorem~\ref{thm:euler-critical-count} applies.
\end{proof}

Corollary~\ref{cor:ball-boundary-jacobian} extends the planar boundary
Jacobian criterion of Alessandrini and Nesi \cite{AlessandriniNesi2009}.
The boundary Jacobian test supplies the containment hypothesis of
Theorem~\ref{thm2} for smooth data. Example~\ref{ex1} gives nonconvex
domains to which the test applies.

\section{Compact holes and round shells}
\label{sec:spherical-shells}

Hartogs' theorem gives a pluriharmonic extension across a compact
hole. To prove that the extension remains injective, we apply
Corollary~\ref{cor:ball-boundary-jacobian} on spheres in the region
where the original map is defined.

\begin{theorem}[Removability of compact holes]\label{thm:removable-hole}
Let $n\ge2$, let $B_R\subset\C^n$ be the ball of radius $R$, and
let $K\subset B_R$ be compact with $B_R\setminus K$ connected.
Every injective pluriharmonic map
$F:B_R\setminus K\to\R^{2n}$ extends uniquely to a pluriharmonic
real-analytic diffeomorphism of $B_R$ onto its image.
\end{theorem}

\begin{proof}
For each real component $u$ of $F$, extend the coefficients of the
holomorphic one-form $2\partial u$ to $B_R$ by Hartogs' extension
theorem \cite{Hormander1990}. The extended form is closed by the
identity theorem, since the original form is closed on the connected
complement of $K$. It has a holomorphic primitive on the ball.
The real part of that primitive agrees with $u$ up to a real
constant on $B_R\setminus K$. These primitives give the unique
pluriharmonic extension $\widetilde F$.

Choose $s<R$ with $K\subset B_s$. The sphere $\partial B_s$ lies
in the region of injectivity. Theorem~\ref{thm1} gives a
nonvanishing Jacobian on this sphere. Hence $\widetilde F(\partial B_s)$
is a smoothly embedded sphere. Let $D_s$ be the bounded component
of its complement. Corollary~\ref{cor:ball-boundary-jacobian}
applied to $\widetilde F|_{\overline B_s}$ gives a diffeomorphism
onto $\overline D_s$. The balls $B_s$ exhaust $B_R$, which proves
injectivity and nondegeneracy of $\widetilde F$ on $B_R$.
\end{proof}

For the shell $A(r,R)=B_R\setminus\overline B_r$, write $S_r=\partial B_r$.

\begin{corollary}[Filling and the sharp radius bound]
\label{cor:round-shell-filling}
Let $n\ge2$. A pluriharmonic homeomorphism
$F:A(r,R)\to A(r_*,R_*)$ extends uniquely to a pluriharmonic
real-analytic diffeomorphism
$\widetilde F:B_R\to B_{R_*}$ with $\widetilde F(B_r)=B_{r_*}$.
In particular,
\begin{equation}\label{eq:round-shell-radius}
       \frac{r_*}{R_*}\le\frac{r}{R}.
\end{equation}
Equality holds precisely when $\widetilde F(z)=(R_*/R)Uz$ for a
real orthogonal transformation $U\in O(2n)$.
\end{corollary}

\begin{proof}
Apply Theorem~\ref{thm:removable-hole} to $K=\overline B_r$. On
each $S_s$ with $r<s<R$, the norm of $F$ has maximum strictly below
$R_*$. The maximum principle applied to the subharmonic function
$|\widetilde F|^2$ on $B_s$ gives $|\widetilde F|<R_*$ throughout
$B_R$. For $x\in S_r$, continuity from the shell gives
$\widetilde F(x)\in\overline{A(r_*,R_*)}$. It cannot lie in
$A(r_*,R_*)$, where every point already has a preimage in
$A(r,R)$. Thus injectivity places $\widetilde F(S_r)$
in $S_{r_*}$. The restriction is an embedding of a sphere into a
sphere of the same dimension. Its image is open by invariance of
domain and closed by compactness, so it fills $S_{r_*}$.
Corollary~\ref{cor:ball-boundary-jacobian} now gives
$\widetilde F(B_r)=B_{r_*}$. The shell, the seam, and the filled
ball together have image $B_{R_*}$.

Let $p=\widetilde F^{-1}(0)\in B_r$. Since $n\ge2$, choose a unit
vector $v$ Hermitian orthogonal to $p$. For $s=r,R$ write
$\rho_s=(s^2-|p|^2)^{1/2}$. The circle $p+\rho_se^{it}v$ lies in
$S_s$. On this complex line write
\[
       \widetilde F(p+\zeta v)=\Repart\sum_{j\ge1} A_j(v)\zeta^j ,
                    \qquad |\zeta|<\rho_R ,
\]
choosing the holomorphic primitive to vanish at $\zeta=0$. Taking
circular means of the squared norm at the two radii gives
\[
     s_*^2=\tfrac12\sum_{j\ge1}|A_j(v)|^2\rho_s^{2j},
     \qquad (s,s_*)=(r,r_*),(R,R_*).
\]
At the outer radius, properness gives the boundary norm limit and
the sum converges by monotone convergence from smaller circles.
Termwise comparison gives
\[
  \frac{r_*^2}{R_*^2}\ \le\ \frac{\rho_r^2}{\rho_R^2}
     =\frac{r^2-|p|^2}{R^2-|p|^2}\ \le\ \frac{r^2}{R^2}.
\]
Equality of the end terms forces $p=0$. Every unit $v$ is then
admissible, and equality in the first inequality forces $A_j(v)=0$
for every $j\ge2$ in every direction. Thus $\widetilde F$ agrees with
its real derivative at $0$ on the ball. Its value has constant norm
$r_*$ on $S_r$, so its linear part is $(r_*/r)U$ with $U\in O(2n)$.
Such maps attain equality. An alternative proof of the radius bound
and its equality case follows from the harmonic mean-width estimates
in \cite[Theorems~1.4 and~1.8]{ZhongKalaj2026MeanWidth}.
\end{proof}

To complete the proof of Theorem~\ref{thm:shell-rationality}, it remains
to exclude strict inequality in \eqref{eq:round-shell-radius}.
Appendix~\ref{app:round-shells} does so by analyzing the rational
holomorphic primitive of the filled map.

\section{Quasiconformal consequences}
\label{sec:metric-notation}

For domains $U,V\subset\R^d$, $d\ge2$, the local Sobolev space
$W^{1,d}_{\mathrm{loc}}(U,\R^d)$ consists of maps whose coordinates and
distributional first derivatives are locally integrable to the power $d$.
For a constant $K\ge1$, a nonconstant continuous map
$f\in W^{1,d}_{\mathrm{loc}}(U,\R^d)$ is $K$-quasiregular if it is
sense preserving and satisfies
\[
 \|Df(x)\|^d\le K\Jac f(x),\qquad \Jac f(x)\ge0,
 \quad\text{for almost every }x\in U,
\]
where almost everywhere refers to $d$-dimensional Lebesgue measure and
the differential is taken in the weak sense. It is $K$-quasiconformal
if, in addition, it is a homeomorphism onto its image. The corresponding
terms without a specified $K$ mean that some finite $K$ exists.
We also allow orientation-reversing quasiconformal maps by composing
with a target reflection; their distortion inequalities use
$|\Jac f|$. The outer dilatation $K_O(f)$ is the least admissible
constant in the outer distortion inequality. The inner dilatation
$K_I(f)$ is the least constant $K'\ge1$ for which
\[
 |\Jac f(x)|\le K'\ell(Df(x))^d
 \quad\text{for almost every }x\in U.
\]

A domain $U$ is quasiconvex if there is a constant $Q_U\ge1$
such that any two points $x,y\in U$ can be joined inside $U$ by
a rectifiable curve of length at most $Q_U|x-y|$. Such a $Q_U$
is a quasiconvexity constant; rectifiable means finite length.
The Lipschitz, co-Lipschitz, and bi-Lipschitz conventions are those
stated in the introduction.

\begin{definition}\label{def:dini}
Let $r_0>0$. A modulus of continuity on $[0,r_0]$ is a continuous
nondecreasing function $\omega:[0,r_0]\to[0,\infty)$ with $\omega(0)=0$.
It is a Dini modulus if
\[
 \int_0^{r_0}\frac{\omega(t)}{t}\,dt<\infty.
\]
A bounded domain has $C^1$-Dini boundary, also called Dini-smooth
boundary, if it admits finitely many boundary charts given by $C^1$
graphs whose gradients share a Dini modulus $\omega$. Thus, within
each chart, the difference of the gradients at two points at distance
$t$ is bounded by $\omega(t)$. We use the same terminology for Jordan
curves in the plane.
\end{definition}

The distortion definitions are those of \cite{Vaisala1971}.
Appendix~\ref{sec:boundary} records the boundary estimates for
harmonic quasiconformal maps onto Dini-smooth domains.

\begin{corollary}\label{cor:bilipschitz-intro}
Let $U$ and $D$ be bounded domains in $\C^n$ with $C^1$-Dini
boundary. Every quasiconformal pluriharmonic homeomorphism
$F:U\to D$ is bi-Lipschitz and extends bi-Lipschitz to the closures.
\end{corollary}

\begin{proof}
Proposition~\ref{thm:qc-direct} supplies interior nondegeneracy
independently of the unrestricted local Lewy theorem. Apply
Theorem~\ref{thm:bilipschitz-criterion} to the real coordinate
functions and extend the two Lipschitz bounds to the completions.
\end{proof}

\subsection{An independent argument under quasiconformal distortion}
\label{sec:qc}

In this section the real dimension is $d=2n$. The following special
case of Theorem~\ref{thm1} also admits a direct proof using
quasiconformal compactness.

\begin{proposition}\label{thm:qc-direct}
Let $F:\Omega\subset\C^n\to\C^n$ be a quasiconformal pluriharmonic
homeomorphism. Then $\Jac F$ is nowhere zero, for every $n\ge1$.
\end{proposition}

\begin{proof}
Choose a quasiconformality constant $K\ge1$ for $F$, and suppose
that there is a point $p\in\Omega$ with $\Jac F(p)=0$.
The outer distortion inequality
\[
 \|DF(x)\|^{2n}\le K|\Jac F(x)|
\]
extends from almost every point to every point because $F$ is real analytic.
Hence $DF(p)=0$. Translate so that $p=0$ and $F(0)=0$.
Choose local holomorphic maps $H$ and $G$ such that $F=H+\overline G$.
Let $k\ge2$ be the first nonzero degree in its Taylor expansion.
Write $H_k$ and $G_k$ for the homogeneous terms of degree $k$ of $H$
and $G$, respectively, and set $P_k=H_k+\overline{G_k}$. Thus
\[
 F(z)=P_k(z)+O(|z|^{k+1}),\qquad P_k\not\equiv0.
\]
Here homogeneity means $H_k(tz)=t^kH_k(z)$ and
$G_k(tz)=t^kG_k(z)$ for every $t\in\C$.
For a positive scale $r$, the rescalings $F_r(z)=r^{-k}F(rz)$ are $K$-quasiconformal homeomorphisms on
expanding balls, and converge locally uniformly to the nonconstant map $P_k$.
The compactness theorem for quasiconformal homeomorphisms
\cite{Vaisala1971}, applied on each fixed ball, makes $P_k$ injective.
But for $\zeta=e^{2\pi i/k}\ne1$, homogeneity gives
\[
 H_k(\zeta z)=H_k(z),\qquad G_k(\zeta z)=G_k(z),
 \qquad P_k(\zeta z)=P_k(z),
\]
contradicting injectivity. Thus $\Jac F(p)\ne0$.
\end{proof}

\begin{corollary}
\label{cor:quasiregular}
Let $n\ge2$. Every nonconstant quasiregular pluriharmonic mapping
$F:\Omega\subset\C^n\to\C^n$ has positive real Jacobian everywhere and is
a local real-analytic diffeomorphism.
\end{corollary}

\begin{proof}
Pluriharmonic mappings are real analytic. The classical smooth
quasiregular local-homeomorphism theorem, in real dimension $2n\ge4$, gives
local injectivity; see \cite[Theorem~7.3]{BojarskiIwaniec1983}.
Near each point, the restriction is therefore a quasiconformal
pluriharmonic homeomorphism. Proposition~\ref{thm:qc-direct} gives
$\Jac F\ne0$, and quasiregularity fixes its sign as positive.
\end{proof}

\begin{remark}
In complex dimension one, the homeomorphism hypothesis cannot be omitted.
For any integer $k\ge2$, the map $z\mapsto z^k$ is a smooth
quasiregular pluriharmonic mapping with a critical point at $0$. In complex dimension at least two, the preceding corollary
combines smooth quasiregular local injectivity with the direct rescaling argument.
\end{remark}

For $a,b$ in a proper domain $D$, let $\gamma$ range over
rectifiable curves in $D$ joining them and let $ds$ be Euclidean
arclength. Their quasihyperbolic distance is
\[
 k_D(a,b)=\inf_\gamma\int_\gamma\frac{ds}{\delta_D(z)}.
\]

In the plane, Manojlovi\'c proved bi-Lipschitz continuity in the
quasihyperbolic metric for harmonic quasiconformal mappings between
proper domains \cite[Theorem~1]{Manojlovic2009}. Differentiating that
metric comparison gives the planar case of the bounds below.

\begin{proposition}[Boundary distance and differential stretch]
\label{prop:qc-stretch}
Let $F:U\to V$ be a $K$-quasiconformal pluriharmonic homeomorphism
between proper domains in $\R^{2n}$, with $U\subset\C^n$.
There are constants $c,C>0$ depending only on $n,K$ such that
\begin{equation}\label{eq:qc-stretch}
 c\frac{\delta_V(F(x))}{\delta_U(x)}
 \le\ell(DF(x))\le\|DF(x)\|
 \le C\frac{\delta_V(F(x))}{\delta_U(x)},
 \qquad x\in U.
\end{equation}
\end{proposition}

\begin{proof}
Fix $x\in U$, put $t=\delta_U(x)$ and $s=\delta_V(F(x))$, and
consider $G(w)=(F(x+tw)-F(x))/s$ on the unit ball.
For $|w|\le1/2$, the segment from $x$ to $x+tw$ has uniformly
bounded quasihyperbolic length in $U$. Quasihyperbolic distortion
\cite{GehringOsgood1979} and the elementary bound
\[
 k_V(y,z)\ge\log\bigl(1+|y-z|/\delta_V(y)\bigr)
\]
give a uniform bound for $G$ on the half ball. The harmonic
gradient estimate gives the upper differential bound.

For a sufficiently small $a>0$ depending only on $n,K$,
$G(B_{1/4}^{2n})$ contains $B_a^{2n}$. Indeed, each point of
$B(F(x),as)$ is joined to $F(x)$ by a segment of quasihyperbolic
length at most $-\log(1-a)$. Apply quasihyperbolic distortion to
$F^{-1}$, then the same elementary lower bound in $U$. The
preimages lie in $B(x,t/4)$ when $a$ is sufficiently small,
since the distortion bound tends to zero with the target distance.

If the normalized minimal stretch were not bounded below, a sequence
of these normalized maps would have $\ell(DG_j(0))\to0$.
Harmonic interior estimates give a subsequential pluriharmonic
limit in $C^1_{\mathrm{loc}}(B_{1/2}^{2n})$. It is nonconstant:
$G_j(0)=0$, and the preimages of a fixed nonzero point of $B_a^{2n}$
have a limit point in $\overline{B_{1/4}^{2n}}$. Quasiconformal
compactness \cite{Vaisala1971} makes this nonconstant limit
quasiconformal and injective, while Proposition~\ref{thm:qc-direct}
makes its derivative invertible at $0$. This
contradicts $\ell(DG_j(0))\to0$ and proves \eqref{eq:qc-stretch}.
\end{proof}

\begin{proposition}[Lipschitz bound on the complex ball]
\label{prop:qc-lipschitz-ball}
Let $n\ge2$ and let $F:\B^n\to\R^{2n}$ be a bounded pluriharmonic
map satisfying $\|DF\|\le\kappa\ell(DF)$ in the ball, for a
constant $\kappa\ge1$. Then for a constant $C_n$ depending only on
$n$,
\begin{equation}\label{eq:complex-tangential-lip}
 \sup_{\B^n}\|DF\|\le C_n\kappa^2
       \operatorname{diam} F(\B^n).
\end{equation}
In particular, the bound holds for quasiconformal maps with
$\kappa$ depending only on the distortion.
\end{proposition}

\begin{proof}
For $0<R<1$, set $F_R(z)=F(Rz)$ and
$M_R=\sup_{\B^n}\|DF_R\|<\infty$. Write
$A=\operatorname{diam} F(\B^n)$. Fix a unit vector $\zeta$ and
a complex unit vector $\eta$ with $\sum_j\eta_j\overline{\zeta_j}=0$.
Such an $\eta$ exists because $n\ge2$. At $s\zeta$,
$0\le s<1$, the complex disk in the $\eta$ direction has radius
$\sqrt{1-s^2}$, since $|s\zeta+w\eta|^2=s^2+|w|^2$.
Write $D_\xi F_R$ for the real directional derivative along $\xi$.
Apply the harmonic gradient estimate on this disk to $F_R$ after
subtracting its value at the center, and to the pluriharmonic
function $D_\zeta F_R$. This gives
\[
 |D_\eta F_R(s\zeta)|\le \frac{C_n A}{\sqrt{1-s^2}},
 \qquad
 |D_\eta D_\zeta F_R(s\zeta)|\le
            \frac{C_n M_R}{\sqrt{1-s^2}} .
\]
At each point, $\|DF_R\|\le\kappa|D_\eta F_R|$.

Take $r_0=1-\epsilon$ for $0<\epsilon\le1/2$. For
$r\le r_0$, the first estimate gives
$\|DF_R(r\zeta)\|\le C_n\kappa A/\sqrt\epsilon$.
For $r\ge r_0$, integrate
$\frac{d}{ds}D_\eta F_R(s\zeta)=D_\eta D_\zeta F_R(s\zeta)$
from $r_0$ to $r$. Since
$\int_{1-\epsilon}^1 (1-s^2)^{-1/2}\,ds\le2\sqrt\epsilon$,
both ranges give
\[
 M_R\le C_n\kappa A/\sqrt\epsilon+
        C_n\kappa M_R\sqrt\epsilon
\]
with a fixed $C_n$. Choose $\sqrt\epsilon$ to be a sufficiently
small fixed multiple of $1/\kappa$ and absorb the last term.
Letting $R\uparrow1$ proves \eqref{eq:complex-tangential-lip}.
\end{proof}

\begin{corollary}\label{cor:qc-convex-no-dini}
If $n\ge2$ and $D\subset\R^{2n}$ is bounded and convex, every
quasiconformal pluriharmonic homeomorphism $F:\B^n\to D$ is
bi-Lipschitz, with a bi-Lipschitz extension to the closures.
Its upper Lipschitz bound depends only on $n$, the distortion and
$\operatorname{diam}D$. Its inverse Lipschitz bound depends only
on $n$, the distortion and $\delta_D(F(0))$.
\end{corollary}

\begin{proof}
Proposition~\ref{prop:qc-lipschitz-ball} gives the upper bound.
For every unit vector $\nu$, the supporting deficit
$h_D(\nu)-\langle F(x),\nu\rangle$ is positive and harmonic.
Harnack's inequality for the real ball in dimension $m=2n$ gives
a lower bound of $2^{-m}(1-|x|)$ times its value at $0$.
Taking the infimum over $\nu$ gives
\[
 \delta_D(F(x))\ge 2^{-m}\delta_D(F(0))(1-|x|).
\]
Proposition~\ref{prop:qc-stretch} therefore bounds $\ell(DF)$
below. Integrating the inverse differential along segments in the
convex domain proves the lower Lipschitz bound, and the two
estimates extend to the closures.
\end{proof}

\begin{corollary}[Smooth sources and convex targets]
\label{cor:qc-smooth-source}
Let $n\ge2$ and let $U\subset\C^n$ be a bounded $C^2$ domain.
Every bounded pluriharmonic map $F:U\to\R^{2n}$ satisfying
$\|DF\|\le\kappa\ell(DF)$ for some constant $\kappa\ge1$
is Lipschitz.
More precisely, let $r>0$ be such that $U$ is covered by
Euclidean balls of radius $r$ contained in $U$, and let $Q_U$
be a quasiconvexity constant for $U$. Then
\begin{equation}\label{eq:smooth-source-lip}
 \operatorname{Lip}(F)
 \le \frac{C_nQ_U\kappa^2}{r}\operatorname{diam}F(U).
\end{equation}
In particular, every quasiconformal pluriharmonic homeomorphism
from $U$ onto a bounded convex domain is bi-Lipschitz and extends
bi-Lipschitz to the closures.
\end{corollary}

\begin{proof}
The uniform interior sphere condition and a tubular neighborhood
of the compact $C^2$ boundary give a radius $r>0$ with the stated
covering property. Indeed, if $t=\delta_U(x)<r$ and $q$ is a
nearest boundary point, then $x=q+t\nu(q)$ belongs to the interior
tangent ball $B(q+r\nu(q),r)$, where $\nu(q)$ is the inward unit
normal. If $\delta_U(x)\ge r$, use $B(x,r)\subset U$.
Applying Proposition~\ref{prop:qc-lipschitz-ball} after translation
and scaling on each ball gives
\[
 \|DF(x)\|\le
 \frac{C_n\kappa^2}{r}\operatorname{diam}F(U),\qquad x\in U.
\]
Bounded $C^1$ domains are quasiconvex \cite{MartioSarvas1979}.
Integration along the curves in the definition of $Q_U$ proves
\eqref{eq:smooth-source-lip}.

Now suppose that $F:U\to D$ is a quasiconformal homeomorphism and
$D$ is bounded and convex. Fix $x_0\in U$.
There is a constant $c_U=c(U,x_0)>0$ such that every positive
harmonic function $u$ on $U$ satisfies
\[
 u(x)\ge c_Uu(x_0)\delta_U(x).
\]
To see this, the centers $q+r\nu(q)$ of the interior tangent balls
lie in a compact subset of $U$. Harnack's inequality on the
connected domain gives a uniform lower bound at those centers
relative to $u(x_0)$. Ball Harnack along the inward radius to
$q+t\nu(q)$ gives a lower bound proportional to $t$.
Harnack's inequality on the remaining compact subset completes
the estimate.

Apply it to the supporting deficits
$u_\nu=h_D(\nu)-\langle F,\nu\rangle$, where $|\nu|=1$.
Since $u_\nu(x_0)\ge\delta_D(F(x_0))$, taking the infimum over
support directions yields
\[
 \delta_D(F(x))\ge
 c_U\delta_D(F(x_0))\delta_U(x).
\]
Proposition~\ref{prop:qc-stretch} therefore bounds $\ell(DF)$
below by a positive constant. Integrating the inverse differential
along segments in $D$ proves the inverse Lipschitz bound.
Both bounds extend to the closures.
\end{proof}

\section{A boundary characterization on the complex ball}
\label{sec:pavlovic}

For a harmonic self-homeomorphism of the disk, Pavlovi\'c characterized
quasiconformality by the boundary argument lift being bi-Lipschitz
and essential boundedness of the periodic Hilbert transform of
its derivative \cite[Theorems~1.1 and~1.2]{Pavlovic2002}.
Astala and Manojlovi\'c proved the Lipschitz bound for harmonic
quasiconformal self-homeomorphisms of a real ball and the
bi-Lipschitz conclusion for harmonic quasiconformal gradient
self-homeomorphisms in real dimension three
\cite[Theorems~1.1 and~2.1]{AstalaManojlovic2015}.
The following pluriharmonic criterion allows a bounded convex target
without any a priori boundary smoothness.

Let $m=2n$ and $S=\partial\B^n$. Let $\sigma$ be surface
measure normalized by $\sigma(S)=1$. For continuous boundary data
$\varphi:S\to\R^{2n}$, define the Poisson extension by
\[
 P[\varphi](x)=\int_S\frac{1-|x|^2}{|x-\eta|^m}
                  \varphi(\eta)\,d\sigma(\eta),\qquad x\in\B^n,
\]
This extension is continuous on $\overline{\B^n}$ with boundary
values $\varphi$. Put $P_r\varphi(\zeta)=P[\varphi](r\zeta)$ for $0\le r<1$.
The boundary trace of a continuous map on $\overline{\B^n}$ is
its restriction to $S$. The space $L^\infty(S)$ consists of
functions essentially bounded with respect to $\sigma$, and
$\mathcal D'(S)$ is the space of distributions, that is, continuous
linear functionals on $C^\infty(S)$ with its usual smooth topology.
The Dirichlet-to-Neumann operator $\Lambda$ is the outward normal
derivative of the Poisson extension on smooth data. On continuous
data we define it distributionally by multiplying each spherical
harmonic of degree $k$ by $k$; these harmonics are restrictions
of homogeneous harmonic polynomials of degree $k$. For distributional
data, $P_r$ is defined by duality with the smooth Poisson kernel.
In particular,
\begin{equation}\label{eq:pav-poisson-dn}
  r\partial_r P_r\varphi=P_r(\Lambda\varphi),
  \qquad
  r\partial_rP_r\varphi\longrightarrow\Lambda\varphi
  \quad\text{in }\mathcal D'(S)\quad(r\uparrow1).
\end{equation}
All boundary operators on vector-valued functions act componentwise.
The metric on the sphere or on the boundary of a target domain is
the restriction of the Euclidean distance.

\begin{theorem}[Pavlovi\'c-type boundary criterion]
\label{thm:pavlovic}
Let $D\subset\R^{2n}$ be a bounded convex domain and let
$\varphi:S\to\partial D$ be a homeomorphism. Suppose that
$F=P[\varphi]$ is pluriharmonic in $\B^n$. The following are equivalent:
\begin{enumerate}[label=\textup{(\roman*)}]
\item $F:\overline{\B^n}\to\overline D$ is bi-Lipschitz.
\item $\varphi$ is bi-Lipschitz and $\Lambda\varphi\in L^\infty(S)$.
\end{enumerate}
If $n\ge2$, these conditions are also equivalent to
quasiconformality of $F$ in $\B^n$ for every bounded convex $D$.
For $n=1$ the equivalence with quasiconformality holds when $D$ has
$C^1$-Dini boundary. Moreover, if
\[
 a|\xi-\eta|\le|\varphi(\xi)-\varphi(\eta)|
       \le b|\xi-\eta|\quad(\xi,\eta\in S),
 \qquad \|\Lambda\varphi\|_{L^\infty(S)}\le M
\]
with $0<a\le b$, then the bi-Lipschitz constants of $F$ can be
chosen solely in terms of $n,a,b,M$.
\end{theorem}

\begin{proof}
Corollary~\ref{cor:convex-no-stein} makes $F$ a homeomorphism of the
closures and a real-analytic diffeomorphism of the interiors.
If (i) holds, the restriction $\varphi$ is bi-Lipschitz. The radial
derivatives in \eqref{eq:pav-poisson-dn} are bounded uniformly in $r$.
Their distributional limit is therefore in $L^\infty(S)$.

Assume (ii). Each infinitesimal real rotation on $S$
commutes with $P_r$ and differentiates $\varphi$ into a bounded weak
derivative. The fields $x_i\partial_{x_j}-x_j\partial_{x_i}$, $i<j$,
control all spherical tangential derivatives at $r\zeta$ uniformly
for $r\ge1/2$. Equation~\eqref{eq:pav-poisson-dn} controls the radial
derivative on the same collar. The interior derivative estimate for
harmonic functions handles the smaller ball. The oscillation of $F$
is bounded by the diameter of the boundary data, which is at most $2b$. Integration along line segments
therefore gives a Lipschitz constant $L\le C_m(b+M)$ on the closed ball.

The positive harmonic function $h_D(\nu)-\langle F,\nu\rangle$
has value at $0$ bounded below using only $m,a,b$.
Indeed, orthogonally projecting $\varphi$ onto $\nu^\perp$ and
applying the Borsuk--Ulam theorem \cite{Matousek2003} gives
antipodal points with identical projections. Their images are
separated by at least $2a$. Thus the width of $D$ in direction
$\nu$ is at least $2a$. In a spherical cap of chordal radius
$a/b$ about a minimizer of $\langle\varphi,\nu\rangle$ the
supporting deficit is at least $a$. Its mean, which is its
Poisson value at $0$, is at least $a$ times the normalized
surface measure of that cap, uniformly in $\nu$.
Write $\sigma$ for normalized surface measure and $\operatorname{cap}(a/b)$
for any cap of chordal radius $a/b$. The Poisson kernel at $x\in\B^n$ is
$(1-|x|^2)/|x-\eta|^m\ge 2^{-m}(1-|x|)$ for $\eta\in S$.
Applying this bound and taking the infimum over the supporting
directions yields the lower bound in
\begin{equation}\label{eq:pav-depth}
  c(1-|x|)\le\delta_D(F(x))\le L(1-|x|),
  \qquad c=2^{-m}a\,\sigma\bigl(\operatorname{cap}(a/b)\bigr)>0.
\end{equation}
Here the distance from a point of $D$ to the boundary equals the
infimum of its distances to the supporting hyperplanes. The upper
bound follows by evaluating $F$ at a closest point of $S$.

Lemma~\ref{lem:pav-common-compact} now gives a uniform positive
lower bound for $\ell(DF)$. Integrating $D(F^{-1})$ along segments
in the convex domain proves the inverse Lipschitz bound, which
extends to $\overline D$ by continuity.

A bi-Lipschitz homeomorphism is quasiconformal in the interior.
For $n\ge2$, the converse follows from
Corollary~\ref{cor:qc-convex-no-dini}. For $n=1$ and a $C^1$-Dini
target it follows from Corollary~\ref{cor:bilipschitz-intro}.
\end{proof}

\begin{lemma}[Nondegeneracy from boundary and depth bounds]
\label{lem:pav-common-compact}
Let $f:\overline{\B^n}\to\overline E$ be a homeomorphism mapping
$\B^n$ onto a bounded domain $E$. Suppose $f$ is pluriharmonic
in $\B^n$ and $L$-Lipschitz on its closure, its trace is
co-Lipschitz with constant $a>0$, and
$\delta_E(f(x))\ge c(1-|x|)$ for some $c>0$ in the ball.
Then $\ell(Df)$ has a positive lower bound depending only on
$n,a,c,L$.
\end{lemma}

The proof is given in Appendix~\ref{app:boundary-compactness}.

\begin{corollary}[Nonconvex targets]\label{cor:pav-nonconvex}
Let $D\subset\R^{2n}$ be bounded and $\varphi:S\to\partial D$
a homeomorphism. Suppose that $F=P[\varphi]$ is pluriharmonic
and $F(\B^n)\subset D$. Then $F$ is bi-Lipschitz on the
closures if and only if the trace is bi-Lipschitz,
$\Lambda\varphi\in L^\infty(S)$, and
\[
  \inf_{x\in\B^n}\frac{\delta_D(F(x))}{1-|x|}>0 .
\]
\end{corollary}

\begin{proof}
Theorem~\ref{thm2} gives a homeomorphism of the closures and a
diffeomorphism of the interiors. Necessity follows by restriction,
\eqref{eq:pav-poisson-dn} and the metric lower bound against every
point of the boundary. Sufficiency follows first from the Lipschitz
upper bound in the proof of Theorem~\ref{thm:pavlovic}, which uses
only the two trace bounds, and then from Lemma~\ref{lem:pav-common-compact}
for the lower differential bound.

For the inverse Lipschitz bound, the two trace bounds and the
Lipschitz upper bound make $D$ quasiconvex. Indeed, for $y,z\in D$
whose segment leaves $D$, let $a,b\in\partial D$ be its first
and last boundary points. The Euclidean distance between their
boundary preimages is at most $|a-b|$ times the inverse Lipschitz
constant of the trace. A shortest great-circle arc between these
preimages has length at most $\pi/2$ times that distance.
Push this arc radially into the ball and take its image under $F$.
Its length is at most $L$ times the arc bound. Points on the
two end portions of the original segment tending to $a$ and $b$
have preimages tending to the respective boundary preimages.
Straight segments in the ball connect them to the ends of the
pushed arc, with lengths tending to zero. Their $F$-images give
the short connecting paths in $D$. Together with the two end
portions this bounds the length of a joining path by a fixed
multiple of $|y-z|$ (if $a=b$, the arc is a point).
Integrating $D(F^{-1})$ along such paths and extending to the
closures proves the result.
\end{proof}

For the example $F_{1/2}$ in Example~\ref{ex1}, both trace
conditions hold, but the boundary distance condition fails.
Along $z_1=is$, $z'=0$, the gap to the upper boundary point in
the same section of $D_{1/2}$ is $(1-s)^2/2$.

Let $T$ be the generator of the Hopf circle action on $S$,
so that $T\varphi(\zeta)=\left.\frac{d}{dt}
\varphi(e^{it}\zeta)\right|_{t=0}$ in distributions. Denote by
$\mathcal H_{\mathrm{Hopf}}$ the periodic Hilbert transform along
these circles, with multiplier $-i\operatorname{sgn}k$ on the
$k$-th Fourier mode and multiplier zero at $k=0$.
For an integer $k$, a $k$-th Hopf mode is a function $\psi$ with
$\psi(e^{it}\zeta)=e^{ikt}\psi(\zeta)$. Define $|T|$ by the
multiplier $|k|$ on these modes.

\begin{proposition}\label{prop:pav-hopf}
For every continuous trace $\varphi$ whose Poisson extension is
pluriharmonic,
\[
  \Lambda\varphi=|T|\varphi=\mathcal H_{\mathrm{Hopf}}(T\varphi)
  \quad\text{in }\mathcal D'(S),
\]
\end{proposition}

\begin{proof}
Put $F=P[\varphi]$. For each $\zeta\in S$, the map $\lambda\mapsto F(\lambda\zeta)$
is harmonic in the disk with boundary data
$t\mapsto\varphi(e^{it}\zeta)$. Its $k$-th Fourier coefficient on
the circle of radius $r$ is $r^{|k|}$ times the corresponding
boundary coefficient. Thus the $k$-th Hopf mode of the outward
derivative is $|k|$ times the boundary mode. Together with
\eqref{eq:pav-poisson-dn}, this identifies $\Lambda$ with $|T|$.
The multipliers of $\mathcal H_{\mathrm{Hopf}}$ and $T$ multiply
to $|k|$.
\end{proof}

\begin{remark}\label{rem:pav-target}
Under either equivalent condition in Theorem~\ref{thm:pavlovic},
the convex target has $C^1$ boundary. In fact, fix $\zeta\in S$ and an
outward support normal $\nu$ at $\varphi(\zeta)$. The function
$g(\eta)=\langle\varphi(\zeta)-\varphi(\eta),\nu\rangle$ on $S$ is
nonnegative and Lipschitz. Since $F$ is Lipschitz,
$P_rg(\zeta)\le L(1-r)$. Fatou's lemma applied to the Poisson
integral divided by $1-r$ shows that
\[
  \int_S\frac{g(\eta)}{|\eta-\zeta|^m}\,d\sigma(\eta)<\infty.
\]
This implies $g(\eta)=o(|\eta-\zeta|)$. Otherwise there are points
at distances $s_j\downarrow0$ from $\zeta$ with
$g(\eta_j)\ge\epsilon s_j$ for some $\epsilon>0$.
By Lipschitz continuity a spherical ball of radius proportional to
$s_j$ about each $\eta_j$ contributes a fixed positive amount to the
integral. A subsequence can be chosen with these balls disjoint,
which contradicts finiteness.

In a smooth sphere chart centered at $\zeta$, rescale the trace
about $\varphi(\zeta)$ and take a locally uniform limit. The limit
is a bi-Lipschitz map on a neighborhood of $0$ in $\R^{m-1}$.
The estimate for $g$ places its image in the hyperplane
perpendicular to every support normal at $\varphi(\zeta)$. Two
independent support normals would place the image in a subspace of
dimension at most $m-2$. Composing with a linear inclusion into
$\R^{m-1}$ contradicts invariance of domain. Hence the
supporting hyperplane is unique at each point. Locally one can
write a convex boundary as the graph of a Lipschitz convex
function. Uniqueness of support makes that function differentiable
throughout the chart, and its gradient is continuous by convexity.
Thus the boundary is $C^1$.
\end{remark}

In complex dimension $n\ge2$, Corollary~\ref{cor:qc-convex-no-dini}
and the preceding argument imply that a bounded convex domain admitting a
quasiconformal pluriharmonic parametrization by the ball has $C^1$ boundary.
In complex dimension one, a conformal map onto a square shows that
the unrestricted quasiconformal implication fails.

\appendix

\section{Proof of the global boundary theorem}\label{app:global-topology}

\subsection{Boundary degree and half-space sections}\label{sec:comparison}

Use $U=\Omega$, $V=D$, $f=F$, and $m=2n$. Set $X=\overline U$,
$Y=\overline V$, $B_X=\partial U$, and $B_Y=\partial V$. All groups
in this proof have coefficients $k=\mathbb F_2$.

Write $H^j$ for singular cohomology in degree $j$. For a compact
set $K$, $\check H^j(K)$ denotes its \v{C}ech
cohomology. For compact subsets of Euclidean space this is the
direct limit of the singular cohomology of shrinking open
neighborhoods. It agrees with singular cohomology on compact
manifolds and finite complexes. For a closed subset $A\subset K$,
the relative group $\check H^j(K,A)$ is the compactly supported
cohomology $H_c^j(K\setminus A)$. More generally, for a locally
compact space $M$, $H_c^j(M)$ is the reduced \v{C}ech cohomology of
$M$ with a point at infinity adjoined, which gives the usual
compactly supported cohomology on manifolds. For a closed subset $L\subset M$,
the exact sequence of the compactified pair gives
\begin{equation}\label{eq:compact-support}
 \cdots\longrightarrow H_c^r(M\setminus L)
 \longrightarrow H_c^r(M)\longrightarrow H_c^r(L)
 \longrightarrow H_c^{r+1}(M\setminus L)\longrightarrow\cdots.
\end{equation}
Proper maps taking closed sets to closed sets and open complements
to open complements induce maps between these sequences
\cite{Bredon1993,Hatcher2002,Iversen1986}.

The restriction $f:U\to V$ is proper since the boundary maps
into $B_Y$. By Alexander duality, $\check H^m(X)=0$, so the
connecting map $\check H^{m-1}(B_X)\to H_c^m(U)$ is onto.
Since $f$ is a homeomorphism on the boundary, naturality makes
$f^*:H_c^m(V)\to H_c^m(U)$ nonzero. Both groups are $k$.
In particular $f$ has mod-$2$ degree one and is onto. The degree
remains one on each target component when $f$ is restricted to
the full inverse image of an open subset of $V$.

For a proper map $g:P\to Q$ of $m$-manifolds with total mod-$2$
degree one on each component of $Q$, duality gives
\begin{equation}\label{eq:degree}
 g_*\operatorname{PD}_P g^*=\operatorname{PD}_Q,
 \qquad \operatorname{PD}_Q:H_c^r(Q)\xrightarrow{\ \cong\ }H_{m-r}(Q).
\end{equation}
Here $\operatorname{PD}_Q$ is the Poincar\'e duality map and $g^*$ is
the pullback on compactly supported cohomology. This is the projection
formula for the fundamental class. In particular $g^*$ is injective
and $g_*$ is surjective, and either is an isomorphism if the other
is. For manifolds with boundary and maps with $g^{-1}(\partial Q)
=\partial P$, the formula holds with $H_{m-r}(Q,\partial Q)$ and
$H_{m-r}(P,\partial P)$ and the degree is checked in the interiors
\cite{Bredon1993,Hatcher2002}.

\begin{lemma}\label{lem1}
The maps induced by $f$ on singular homology and compactly supported
cohomology of the interiors, and on \v{C}ech cohomology of the
closures, are isomorphisms.
\end{lemma}

\begin{proof}
By \eqref{eq:degree}, $f_*$ is onto on interior homology and $f^*$ is
injective on interior compactly supported cohomology. The latter
is the comparison on the relative groups of the compact pairs
$(Y,B_Y)$ and $(X,B_X)$. Since the boundary comparison is an
isomorphism, exactness of the pair sequences gives injectivity of
$f^*:\check H^j(Y)\to\check H^j(X)$. Indeed, a class in its kernel
lifts to a relative class, and the lift can be adjusted by a
connecting image from the boundary so that it also pulls back to zero.

For $W=U$ or $V$, cover the sphere $S^m$ (the one-point compactification
of $\R^m$) by $\overline W$ and $S^m\setminus W$, whose intersection is
$\partial W$. Mayer--Vietoris for this compact cover and Alexander
duality for $S^m\setminus W$ give
\begin{equation}\label{eq:rank}
 \dim_k\widetilde{\check H}^{\,j}(\partial W)
 =\dim_k\widetilde{\check H}^{\,j}(\overline W)
  +\dim_k\widetilde H_{m-j-1}(W),
 \qquad 0\le j\le m-2 .
\end{equation}
The dimensions for $W=U$ are finite. Indeed, the compact manifold $X$
is a Euclidean neighborhood retract and hence is dominated by a
finite polyhedron, and the boundary collar gives $U\simeq X$
\cite{Bredon1993,Brown1962}. Both summands on the right for $W=U$
are at least their counterparts for $W=V$, by injectivity on closures
and surjectivity on interiors. The boundary homeomorphism makes their
sums equal. Thus $f_*$ is an isomorphism in degrees $1,\ldots,m-1$,
and the remaining degrees follow from connectedness and the vanishing
in degrees at least $m$ for open $m$-manifolds. Formula~\eqref{eq:degree}
and the two compact pair sequences give the remaining isomorphisms.
\end{proof}

\begin{proposition}[Comparison of closed sections]\label{prop2}
For every closed affine half-space $H\subset\R^m$, $f$ induces
an isomorphism
\[
 \check H^\ast(Y\cap H)\longrightarrow\check H^\ast(f^{-1}(H)).
\]
\end{proposition}

\begin{proof}
For any open half-space $G$, Lemma~\ref{lem4} gives the injections
\begin{equation}\label{eq:upper-inclusion}
 H_j(U\cap f^{-1}(G))\longrightarrow H_j(U),\qquad j\ge n .
\end{equation}
The degree gives a surjection
$H_j(U\cap f^{-1}(G))\to H_j(V\cap G)$ in every degree. In degrees
$j\ge n$ it is also injective, since the map in
\eqref{eq:upper-inclusion} factors through it by
Lemma~\ref{lem1} and naturality.

Write $H=\{\ell\le t\}$, with $\ell$ a nonzero real linear functional.
Put $K_X=f^{-1}(H)$ and $K_Y=Y\cap H$. For $T=X,Y$, write
$C_T=K_T\cap B_T$, $L_T=K_T\setminus C_T$, and $W_T=T\setminus(B_T\cup L_T)$.
Thus $L_T$ is the closed section of the interior, and $W_T$ is
the open section on the other side of the cutting plane.

The restriction $\check H^j(X)\to\check H^j(K_X)$ is onto for $j\ge n$.
To see this, take the shrinking open neighborhoods
$\{x\in X:\ell(f(x))<s\}$ with $s\downarrow t$. Each is homotopy
equivalent to its interior portion by pushing boundary points slightly
inward within the neighborhood along the collar \cite{Brown1962}.
For these interior portions, surjectivity of restriction follows
from \eqref{eq:upper-inclusion} by the universal coefficient theorem
over $k$. Pass to the neighborhood limit. Lemma~\ref{lem1} then gives
surjectivity of $\check H^j(K_Y)\to\check H^j(K_X)$ for $j\ge n$.
The homeomorphism $C_X\to C_Y$ and the exact pair sequences give
surjectivity in the same degrees on the relative groups, that is,
on $H_c^j(L_Y)\to H_c^j(L_X)$.

First take $t$ to be a regular value of $\ell\circ f$ in $U$.
Then $L_X,L_Y$ are manifolds with boundary along the level $t$.
The map between them has degree one on each target interior component,
so \eqref{eq:degree} for manifolds with boundary gives injectivity on
$H_c^j$ as well. In \eqref{eq:compact-support} for the closed sections
$L_X\subset U$ and $L_Y\subset V$, comparison at degrees $r-1,r$ on the
closed sections and on the ambient interiors now gives an isomorphism
on $H_c^r(W_Y)\to H_c^r(W_X)$ for $r\ge n+1$. By \eqref{eq:degree},
this gives the homology isomorphisms for the open sections in degrees
$j\le n-1$. The higher degrees were obtained in
\eqref{eq:upper-inclusion} and the argument following it.
Consequently the comparison for $W_T$ is an isomorphism on $H_c^\ast$
in every degree. By \eqref{eq:compact-support}, so is the comparison
for $L_T$. The exact sequences of $(K_T,C_T)$ give the claim for $K_T$.

For an arbitrary $t$, take regular values decreasing to $t$, which
exist by Sard's theorem. The compact sections decrease to the
desired sections, and the isomorphisms pass to the limit by
continuity of \v{C}ech cohomology.
\end{proof}

\subsection{Point fibers and injectivity}\label{sec:fibers}

\begin{proposition}[Recovery of a point fiber]\label{prop3}
For every $y\in V$, the fiber $f^{-1}(y)$ has the singular
cohomology of a point over $k$.
\end{proposition}

\begin{proof}
Fix $y\in V$. For $T=X,Y$ set $f_X=f$, $f_Y=\operatorname{id}_Y$
and $A_T=f_T^{-1}(y)$. Write $B$ for the closed unit ball in $\R^m$
and $S=\partial B$. The space of points paired with half-space
directions is
\[
 Z_T=\{(x,e)\in T\times S:\ \langle f_T(x)-y,e\rangle\le0\}.
\]
This use of incidence spaces is related to \cite{Schapira1995}.
The map $(x,e)\mapsto(f(x),e)$ from $Z_X$ to $Z_Y$ is onto and
preserves the direction. At each fixed direction the map on
cohomology is the isomorphism of Proposition~\ref{prop2}.
Bia\l ynicki-Birula's theorem \cite[Theorem~$1'$]{BialynickiBirula1964}
therefore gives an isomorphism $\check H^\ast(Z_Y)\to\check H^\ast(Z_X)$.

Now fill the sphere of directions above each point of the fiber:
\[
 W_T=Z_T\cup(A_T\times B)\ \subset T\times B.
\]
The projection $W_T\to T$ has a closed hemisphere over every point
outside $A_T$ and the full ball over every point of $A_T$. Both
are contractible. Applying the same comparison theorem to this
projection over $T$ gives $\check H^\ast(T)\cong\check H^\ast(W_T)$.
By Lemma~\ref{lem1} the comparison between $W_Y$ and $W_X$ is an
isomorphism, so the pair sequences give an isomorphism of
$\check H^\ast(W_Y,Z_Y)\to\check H^\ast(W_X,Z_X)$.

Collapsing $Z_T$ to a point in $W_T$ identifies the resulting
quotient with $(A_T\times B)/(A_T\times S)$. The fiber $A_X$ is a
compact real-analytic set and admits a finite triangulation
\cite{Lojasiewicz1964}. For $A_Y=\{y\}$ the same conclusion is
immediate. The relative product formula \cite[Section~3.2]{Hatcher2002}
therefore gives
\[
 \check H^{j+m}(W_T,Z_T)\cong H^j(A_T;k),\qquad j\ge0,
\]
since the relative cohomology of $(B,S)$ is $k$ in degree $m$
and zero in other degrees. The relative comparison just proved
now gives the stated cohomology of $A_X$.
\end{proof}

\begin{lemma}\label{lem3}
A nonempty compact real-analytic set with the singular homology
of a point over $\mathbb F_2$ consists of one point.
\end{lemma}

\begin{proof}
By the fundamental-class theorem of Borel and Haefliger
\cite[Theorem~3.7]{BorelHaefliger1961}, a compact real-analytic set of
positive dimension has nonzero top homology over $\mathbb F_2$.
Thus the set is zero-dimensional and finite. Its $H_0$ has
dimension one, so it consists of a single point.
\end{proof}

\begin{proof}[Proof of Theorem~\ref{thm2}]
Every interior fiber is nonempty by degree one and compact by
properness. Proposition~\ref{prop3} and the universal coefficient
theorem give the singular homology of a point for each such fiber.
By Lemma~\ref{lem3}, each fiber is a singleton. Together with the
boundary homeomorphism and $F(\Omega)\subset D$ this makes $F$ a
continuous bijection of the compact closures, hence a homeomorphism.
Theorem~\ref{thm1} gives the real-analytic interior inverse.
\end{proof}

\subsection{A boundary compactness argument}
\label{app:boundary-compactness}

\begin{proof}[Proof of Lemma~\ref{lem:pav-common-compact}]
The upper depth bound is $\delta_E(f(x))\le L(1-|x|)$.
Suppose there were no uniform lower bound for $\ell(DF)$ in terms
of $n,a,c,L$. Take a sequence of maps $f_j$ with those bounds,
targets $E_j$, and points $x_j$ such that
$\ell(Df_j(x_j))\to0$. Put $t_j=1-|x_j|$, and choose
$p_j=x_j/|x_j|$ when $x_j\ne0$ and any unit vector otherwise.
Choose a unitary map $U_j$ with $U_je_1=p_j$. Define
\[
 \begin{gathered}
 V_j=\{w:p_j+t_jU_jw\in\B^n\},\qquad D_j=t_j^{-1}(E_j-f_j(x_j)),\\
 F_j(w)=\frac{f_j(p_j+t_jU_jw)-f_j(x_j)}{t_j}.
 \end{gathered}
\]
Passing to a subsequence, $t_j\to\tau\in[0,1]$. Each compact
subset of $V=\{2\Repart w_1+\tau|w|^2<0\}$ lies in all
sufficiently late $V_j=\{2\Repart w_1+t_j|w|^2<0\}$.
The ball or half-space $V$ contains $-e_1$, and $F_j(-e_1)=0$.
Each $F_j$ is a homeomorphism onto $D_j$. The uniform Lipschitz
bound and harmonic interior estimates give a subsequence converging
in $C^\infty_{\mathrm{loc}}(V)$ to a pluriharmonic map $G$,
with $G(-e_1)=0$ and $\det DG(-e_1)=0$.

We claim that the inverse images $F_j^{-1}(\overline B(0,c/2))$
lie in one compact subset of $V$ for all sufficiently large $j$.
By the depth bounds, $c\le\delta_{D_j}(0)\le L$, so
$\overline B(0,c/2)\subset D_j$. Scaling the same estimate gives,
if $|F_j(w)|\le c/2$,
\[
  \frac{c}{2L}\le\delta_{V_j}(w)
       \le \frac{L+c/2}{c}.
\]
For $x=p_j+t_jU_jw$, take a closest point $\xi\in S$, and write $\varphi_j=f_j|_S$. Then
\[
  |\varphi_j(\xi)-\varphi_j(p_j)|
  \le L(1-|x|)+(c/2)t_j+Lt_j.
\]
The inverse Lipschitz estimate for $\varphi_j$ therefore bounds
$|\xi-p_j|$ by a constant times $t_j$. Since $1-|x|=O(t_j)$,
this bounds $|w|$ uniformly. On bounded sets in $V_j$ the
identity
\[
 \delta_{V_j}(w)
 =\frac{-2\Repart w_1-t_j|w|^2}{1+|e_1+t_jw|}
\]
shows that the uniform positive lower bound for the boundary
distance keeps the inverse images in a compact subset of $V$.

Fix $y\in B(0,c/2)$. A subsequential limit of the points
$F_j^{-1}(y)$ belongs to $A=G^{-1}(y)$. Conversely, for every $w\in A$
one has $|F_j(w)|<c/2$ for all sufficiently large $j$, so $w$
belongs to the common compact subset of $V$. Thus $A$ is nonempty
and compact.
Every open neighborhood $\mathcal U$ of $A$ in $V$ contains
$F_j^{-1}(\overline B(y,\varepsilon))$ for some sufficiently small,
fixed $\varepsilon>0$ and all large $j$. Otherwise, taking
$\varepsilon\downarrow0$ and $j\to\infty$ gives inverse images
outside $\mathcal U$ with a limit point in $A$, by the compact
containment. Once $\varepsilon$ is fixed, uniform convergence on
$A$ implies $A\subset F_j^{-1}(B(y,\varepsilon))$ for all large
$j$. The compact real-analytic set $A$ is triangulable, and thus is a
neighborhood retract \cite{Lojasiewicz1964,Bredon1993}.
Let $r:O\to A$ be a retraction from an open neighborhood $O$
of $A$ in $V$. The preceding argument supplies a set
$N=F_j^{-1}(\overline B(y,\varepsilon))$ with $A\subset N\subset O$.
This set is contractible since $F_j$ is a homeomorphism.
Composing a contraction of $N$ restricted to $A$ with $r$
contracts $A$. Lemma~\ref{lem3} now shows that $A$ is a singleton.

This holds for every $y\in B(0,c/2)$, hence $G$ is locally
injective at $-e_1$. Theorem~\ref{thm1} contradicts
$\det DG(-e_1)=0$. This proves the uniform lower bound.

\end{proof}

\section{Rationality and rigidity on spherical shells}
\label{app:round-shells}

We prove Theorem~\ref{thm:shell-rationality}. Write $F_0=\widetilde F(0)$
and choose the holomorphic primitive $P:B_R\to\C^{2n}$ of $\widetilde F$
so that $P(0)=F_0\in\R^{2n}$. Set $D=P'(0)$ and $Q=P\cdot P$, where
$x\cdot y=\sum x_jy_j$ is complex bilinear. We abbreviate $x\cdot x$
to $x^2$. Hermitian orthogonality refers to $x\cdot\bar y$, and
$D^\dagger$ denotes the Hermitian adjoint. Since the real derivative
$v\mapsto\Repart(Dv)$ is invertible, $D$ has complex rank $n$.
For a rational function or vector $g$, put $g^*(w)=\overline{g(\bar w)}$.

\begin{lemma}[A rational primitive]\label{lem:shell-primitive}
There are polynomials $a=(a_1,\ldots,a_{2n}),b,c$ with $(P,Q)=(a,c)/b$
and $b$ nowhere zero on $B_R$. Their degrees are at most
$\lfloor n(2n+1)/(2n-3)\rfloor$.
\end{lemma}
\begin{proof}
Put $m=2n$.
For $q=2r_*^2-Q$ we have
\[
 \Repart q-|P|^2=2(r_*^2-|\widetilde F|^2).
\]
The map
\[
 G=\left(\frac{2P}{1+q},\frac{1-q}{1+q}\right)
\]
is a proper holomorphic map from $B_r$ into $\mathbb B^{m+1}$,
since $1-|G|^2=4(\Repart q-|P|^2)/|1+q|^2$ and the right-hand
side is positive on $B_r$ and zero on $S_r$.
The inequality $\Repart q\ge0$ on $\overline B_r$ ensures that
$G$ extends holomorphically to a neighborhood of that closed ball.
After scaling the source, Forstneri\v c's theorem
\cite[Theorem~1.4]{Forstneric1989} gives rationality, and the degree bound of
D'Angelo and Lebl \cite[Corollary~2]{DAngeloLebl2009}
is $(m+1)m/(2(2n-3))$. The inverse fractional linear formulas
$P=G'/(1+G_{m+1})$ and $q=(1-G_{m+1})/(1+G_{m+1})$ preserve a
common degree bound for $(P,Q)$. These identities extend to $B_R$
by clearing denominators. Cancel common polynomial factors in
the representation $(P,Q)=(a,c)/b$. Because $P,Q$ are holomorphic
on $B_R$, no irreducible factor of $b$ meets $B_R$. Indeed,
generic points of such a factor in the ball would force that
factor to divide all the numerators.
\end{proof}

Let $\phi(t)=\alpha+\beta t$ interpolate $(r^2,r_*^2)$ and $(R^2,R_*^2)$.
Then $\beta>0$. By Corollary~\ref{cor:round-shell-filling} it suffices to
exclude $r_*/R_*<r/R$, or equivalently $\alpha<0$. Assume this strict
inequality until the end of the appendix. With $u=z\cdot w$ put
\begin{align}
 H(z,w)={}&c(z)b^*(w)+b(z)c^*(w)+2a(z)\cdot a^*(w)
             \notag\\[-2pt]
          &-4(\alpha+\beta u)b(z)b^*(w).
 \label{eq:shell-kernel}
\end{align}
On $w=\bar z$ this is $4|b|^2(|\widetilde F|^2-\phi(|z|^2))$.
It vanishes on $S_r$ and $S_R$. At the outer sphere this follows by
continuity of the polynomial and the limit $|\widetilde F(z)|\to R_*$
as $|z|\to R$, since the filled map is proper. The polynomial
divisibility criterion for a sphere
\cite[discussion preceding Lemma~2.5]{AlHelalLeblNandiAnnuli},
followed by complexification, gives
\begin{equation}\label{eq:shell-two-factors}
             (u-r^2)(u-R^2)\mid H(z,w).
\end{equation}
The factors are relatively prime. In particular, for
$(s,s_*)=(r,r_*)$ or $(R,R_*)$ we have the rational identity
\begin{equation}\label{eq:shell-segre}
       (P(z)+P^*(w))^2=4s_*^2\qquad (z\cdot w=s^2).
\end{equation}

For $n\ge4$ the coefficient matrix of $H$ in the two monomial bases
has rank at most $2+2n+n$. For the first two rank-one terms one may
absorb $-4\alpha b b^*$ into $c b^*+b c^*$. A nonzero polynomial
with the factors in \eqref{eq:shell-two-factors} has coefficient rank at
least $\binom{n+2}{2}$ by \cite[Theorem~4.1]{AlHelalLeblNandiAnnuli}.
This exceeds $3n+2$ for $n\ge4$, so $H=0$. But
$H(0,0)=4|b(0)|^2(|F_0|^2-\alpha)>0$. This settles these dimensions.

\subsection{The leading term of the rational primitive}

\begin{lemma}\label{lem:shell-leading}
Under the strict inequality, let $e=\deg b$ and $d=\max_j\deg a_j$.
Then $d=e+1$ and $\deg c\le d$. Write $a_d,b_e$ for the corresponding
homogeneous parts. Set $L=a_d/b_e$ and $E(z)=P(z)-F_0-Dz$. Then
\begin{align}
 &D^\dagger D<2\beta I,\qquad L^2=0,\qquad D^\dagger L(z)=2\beta z,
       \label{eq:shell-linear-data}\\
 &E(z)\cdot L^*(w)\equiv0\pmod{u^2},\qquad
   L(z)\cdot L^*(w)-2\beta u\equiv0\pmod{u^2}. \label{eq:shell-contact}
\end{align}
Congruences of rational functions mean divisibility after clearing
denominators depending separately on $z$ and $w$. The complex span
of the values of $L$ has dimension greater than $n$.
\end{lemma}

\begin{proof}
For $|v|=1$ write $P(\zeta v)=F_0+\sum_{j\ge1}p_j(v)\zeta^j$.
The circular mean of $|\widetilde F(\zeta v)|^2$ at radius $s<R$ is
$|F_0|^2+\frac12\sum_{j\ge1}|p_j(v)|^2s^{2j}$. At $R$ use the
boundary norm limit and monotone convergence of the sum. Subtracting
the two boundary means gives
\[
  2\beta=|Dv|^2+
       \sum_{j\ge2}\frac{R^{2j}-r^{2j}}{R^2-r^2}|p_j(v)|^2 .
\]
If $|Dv|^2=2\beta$ for any $v$, its inner boundary mean would give
$\alpha=|F_0|^2$. This proves the strict matrix inequality.

In \eqref{eq:shell-segre} take $z=\tau v$, $w=(s^2/\tau)\bar v$
for a generic unit vector $v$. Here and below generic directions
exclude the zeros of the finitely many nonzero polynomials needed
in the calculation. As $\tau\to\infty$, subtract the inner sphere
identity from the outer one to obtain
\[
  4(R_*^2-r_*^2)=
  2P(\tau v)\cdot
       \bigl((R^2-r^2)\overline{Dv}/\tau+O(\tau^{-2})\bigr)
       +O(\tau^{-1}).
\]
Since $P(\tau v)$ has order $\tau^{d-e}$, this forces $d-e\ge1$.
If $d-e\ge2$, the leading coefficient gives
$a_d(v)\cdot\overline{Dv}=0$. Homogeneity and comparison of the
coefficients in $v,\bar v$ put every coefficient vector of $a_d$ in
$(\operatorname{im}D)^\perp$. Their span has dimension at most $n$.

To see the contradiction, average \eqref{eq:shell-kernel} under
$(z,w)\mapsto(e^{it}z,e^{-it}w)$. This retains the terms of equal
bidegree and preserves both factors in \eqref{eq:shell-two-factors}.
The largest remaining bidegree is $(d,d)$ with term
$2a_d(z)\cdot a_d^*(w)$, since in $cb^*+bc^*$ an equal bidegree is at
most $(e,e)$, and $u b b^*$ has degree at most $(e+1,e+1)$.
Thus $|a_d(z)|^2$ is divisible by $|z|^4$. Its coefficient rank is
the dimension of its coefficient span, whereas a nonzero multiple
of $|z|^4$ has rank at least $\binom{n+1}{2}>n$
\cite[Theorem~2.1]{DAngeloLeblPfister2012}. Hence $d=e+1$.
Now \eqref{eq:shell-segre} along the same line gives $Q(\tau v)=O(\tau)$,
so $\deg c\le d$ (allowing $c=0$). The identity $a^2=bc$ in degree
$2d$ gives $a_d^2=0$.

Let $c_d$ be the part of degree $d$, possibly zero. The part of $H$
of degree $d$ in $w$ is
\[
 b(z)c_d^*(w)+2a(z)\cdot a_d^*(w)-4\beta u\,b(z)b_e^*(w).
\]
By \eqref{eq:shell-two-factors} it is divisible by $u^2$, as one sees by taking the
top degree in $w$ of the product there. Divide by $b(z)b_e^*(w)$
and take the constant and linear terms at $z=0$. This gives,
respectively, $c_d^*/b_e^*=-2F_0\cdot L^*$ and
$Dz\cdot L^*(w)=2\beta u$, proving the last identity
in \eqref{eq:shell-linear-data} and the first in \eqref{eq:shell-contact}.
The part of bidegree $(d,d)$ is
$2a_d(z)\cdot a_d^*(w)-4\beta u b_e(z)b_e^*(w)$ and is also
divisible by $u^2$. This proves the other congruence. The bilinear
polynomial $u$ is irreducible for $n\ge2$ and has no factor in just
one set of variables, so these divisions preserve the congruences.

If the values of $L$ span a space of dimension at most $n$, the
restriction of $D^\dagger$ to it is an isomorphism onto $\C^n$ by
\eqref{eq:shell-linear-data}. Then $L$ is linear and the second
congruence in \eqref{eq:shell-contact}, of bidegree $(1,1)$, gives
$L^\dagger L=2\beta I$. For $v\ne0$ this implies
$2\beta|v|^2=|L(v)\cdot\overline{Dv}|\le\sqrt{2\beta}|v|\,|Dv|$,
contrary to the first inequality in \eqref{eq:shell-linear-data}.
\end{proof}

\subsection{Complex dimension three}

\begin{lemma}\label{lem:shell-pole}
Suppose $n=3$ and \eqref{eq:shell-linear-data}--\eqref{eq:shell-contact}
hold. Then the reduced homogeneous denominator of $L$ is a linear form.
\end{lemma}

\begin{proof}
Put $V=\operatorname{im}D$ and decompose $L(z)=Tz+K(z)$, where
\[
 T=2\beta D(D^\dagger D)^{-1},\qquad K(z)\in V^\perp.
\]
Indeed $D^\dagger(L-Tz)=0$. The strict bound for $D^\dagger D$ gives
$T^\dagger T=(2\beta)^2(D^\dagger D)^{-1}>2\beta I$. Let
$S=(T^\dagger T-2\beta I)^{1/2}$ and define
$J(z)=(Sz,K(z))\in\C^3\oplus V^\perp$. Its kernel
$J(z)\cdot J^*(w)=L(z)\cdot L^*(w)-2\beta u$ vanishes to order at
least two along $u=0$, away from poles. Fix a generic $v$ and let
$\Pi=v^{\perp}$ be its Hermitian orthogonal plane. Differentiation in
$w$ at $\bar v$ shows that
\[
   W_\Pi:=\operatorname{span}\{J(z):z\in\Pi,\ J(z)\text{ defined}\}
       \ \perp\ \operatorname{im}dJ_v .
\]
The derivative has rank three by its first block, so $\dim W_\Pi\le3$.
The first projection of $W_\Pi$ is $S\Pi$, of dimension two.
Consequently $W_\Pi\cap(0\oplus V^\perp)$ has dimension at most one.
Both $W_\Pi$ and $\operatorname{im}dJ_v$ belong to the global
span of $J$ (for the derivatives, take differences and limits).
Their orthogonality shows that this global span has dimension at
least five. In particular, for no fixed line $U\subset V^\perp$ can
$K\bmod U$ be linear. Otherwise the graph of the two linear blocks
modulo $U$ would have dimension three, and its inverse image would
have dimension at most four.

Interpret the homogeneous rational map $K$ of degree one in affine
charts of $\mathbb P^2$, by choosing a nonzero linear coordinate
as local trivialization. On the projective line $\mathbb P\Pi$
every coefficient of a negative power in a Laurent expansion of $K$
takes its value in the same target line, if the restriction has poles.
Indeed the Laurent coefficients of $J$ belong to $W_\Pi$ and the first
block has no pole. A change of regular trivialization preserves this
property.

We claim that the support of the pole divisor is one projective line.
If it were empty, clearing a reduced homogeneous denominator would
make $K$ linear. Suppose it has two distinct irreducible components,
or one component of degree at least two. At a generic smooth point $x$
of a component let $U(x)$ be the line spanned by the leading polar
coefficient. This direction is unchanged on restriction to any
transverse line through $x$. For generic points $x,y$ on two chosen
distinct components, their joining line is generic, so $U(x)=U(y)$.
The same argument applies to a pair of generic distinct points on one
nonlinear component. To justify genericity, the map sending the pair
to its joining line is dominant in each case, since a general line meets
both components transversely, or meets the one nonlinear component in
at least two distinct smooth points. Its inverse image of the set of
admissible lines is therefore dense in the product of the components.
It follows that $U(x)$ is a fixed line $U$ on a dense open subset of
each component, common to all of them. On a generic transverse
line through a generic pole all principal-part coefficients then lie
in $U$. Thus $K\bmod U$ has no divisor of poles, since a pole at the generic
point of any component would persist on a generic transversal. Its
reduced homogeneous denominator must be constant, so $K\bmod U$ is
linear, a contradiction.

It remains to check the order $m$ of the single pole line. Use affine
coordinates $(s,t)$ with that line given by $t=0$. At a general smooth
point write
\[
 K(s,t)=t^{-m}\big(h(s)+t h_1(s)+O(t^2)\big),\qquad h(s)\ne0.
\]
If $m\ge2$, the transversal $s=s_0+\eta t$ has first two polar
coefficients $h(s_0)$ and $h_1(s_0)+\eta h'(s_0)$. Collinearity for
generic slopes $\eta$ gives $h'(s_0)\in\C h(s_0)$. The projective
direction of $h$ is therefore constant (differentiate its component
ratios). The preceding transversal argument again makes $K$ linear
modulo a fixed line. Thus $m=1$. The same polar divisor occurs for $L=K+Tz$.
\end{proof}

\begin{proposition}\label{prop:shell-three}
The strict inequality $r_*/R_*<r/R$ is impossible for $n=3$.
\end{proposition}

\begin{proof}
Choose unitary source coordinates $z=(t,x)\in\C\oplus\C^2$ with pole
line $t=0$, and write $A=tL$. It is homogeneous quadratic. The second
congruence in \eqref{eq:shell-contact} now has bidegree $(2,2)$
after multiplication by $t\bar t$, so for some real $\gamma$,
\begin{equation}\label{eq:shell-quadratic}
    |A(t,x)|^2-2\beta|t|^2|z|^2=\gamma |z|^4 .
\end{equation}
Moreover $\gamma>0$ since the left side is $|t|^2|J(z)|^2$ where
$t\ne0$ and $S$ is invertible. In the notation of
Lemma~\ref{lem:shell-pole}, the projection of $A$ onto $V$ is $tTz$.
Write $C(x)=A(0,x)\in V^\perp$. Its three quadratic coefficient
vectors span $V^\perp$, since $|C(x)|^2=\gamma |x|^4$ has diagonal
Gram matrix with three positive entries. Comparing the coefficients
of $t^2\bar x_i\bar x_j$ and $tx_k\bar x_i\bar x_j$ in
\eqref{eq:shell-quadratic} makes the coefficients of $A$ involving
$t$ orthogonal to all coefficients of $C$. Hence $A=tTz+C(x)$ and
\[
 |Tz|^2=(2\beta+\gamma)|t|^2+(2\beta+2\gamma)|x|^2 .
\]
In particular $Te_1\perp T(e_1^\perp)$ and $Te_1$ is a nonzero
multiple of $De_1$, by the formula for $T^\dagger T$ in the lemma.

After multiplication by the separate denominators, the first
congruence in \eqref{eq:shell-contact} says $b(z)E(z)\cdot A^*(w)$
is divisible by $(z\cdot w)^2$. At $z=\zeta e_1$ in $B_R$, this
implies that the quadratic $E(\zeta e_1)\cdot A^*(w)$ is a multiple
of $w_1^2$. The other coefficients therefore make $E(\zeta e_1)$
orthogonal to $V^\perp$ and to $T(e_1^\perp)$. Thus
$P(\zeta e_1)-F_0\in\C Te_1$. Since $A^2=0$, we have $(Te_1)^2=0$.
The real and imaginary parts of $Te_1$ are consequently orthogonal
and of equal nonzero length. In a suitable Euclidean complex
coordinate on the resulting real affine plane, $\widetilde F(\zeta e_1)$
is an injective holomorphic function $f(\zeta)$.

Take as coordinate origin the foot of the perpendicular from $0$ to
the plane, and denote its distance from $0$ by $\delta$. Then
$|f(\zeta)|^2=|\widetilde F(\zeta e_1)|^2-\delta^2$.
At radius $r$ its modulus is $\rho_r=\sqrt{r_*^2-\delta^2}>0$,
by injectivity, and at $R$ its modulus tends to
$\rho_R=\sqrt{R_*^2-\delta^2}>\rho_r$. Throughout $r<|\zeta|<R$
it is nonzero since $|\widetilde F|>r_*$. The harmonic function
$\log|f|$ has constant values on both boundary circles, in the
sense of limits. Uniqueness for the annulus gives
$\log|f|=a\log|\zeta|+b_0$, with $a>0$. Thus $\zeta f'/f=a$,
and the winding number on a circle makes $a$ an integer.
Injectivity gives $a=1$. It follows that $\rho_s^2$ is proportional
to $s^2$ at $s=r,R$, giving $\alpha=\delta^2\ge0$, a contradiction.
\end{proof}

\subsection{Complex dimension two}

\begin{lemma}\label{lem:shell-null}
Suppose $n=2$ and $\alpha<0$. Then $F_0=0$ and
\[
                P(z)^2=P(z)\cdot Dz=(Dz)^2=0 .
\]
\end{lemma}

\begin{proof}
For $y$ away from $b_e=0$, the rational map has the expansion
\[
 P(\tau y)=\tau L(y)+C(y)+\tau^{-1}B(y)+O(\tau^{-2}).
\]
The expansion is holomorphic in $y$ near each such point, by division
of the polynomials in $\tau^{-1}$. For a generic $v\in\C^2$ choose
$p\ne0$ spanning $\{p:p\cdot v=0\}$, and $q$ with $q\cdot v=1$.
Set $\ell=L(p)$, $A=dL_p(q)$, $X(\sigma)=P^*(\sigma v)$ and
$Y(\sigma)=X(\sigma)+C(p)$. In \eqref{eq:shell-segre} take
\[
             w=\sigma v,\qquad z=\tau p+tq,\qquad t=s^2/\sigma .
\]
Here the identities first hold for generic nonzero $\sigma$ and
then as rational identities. Expand at $y=p+tq/\tau$. Since $L^2=0$,
we have $\ell\cdot A=0$ and $A^2+\ell\cdot d^2L_p(q,q)=0$.
The coefficient of $\tau$ now gives $\ell\cdot Y=0$, and the constant
term gives
\[
 Y^2+2\ell\cdot B(p)+2t\big(A\cdot Y+\ell\cdot dC_p(q)\big)=4s_*^2 .
\]
Using both radii yields
\begin{equation}\label{eq:shell-quotient}
 \ell\cdot Y=0,\qquad Y^2=4\alpha-2\ell\cdot B(p),\qquad
 A\cdot Y=2\beta\sigma-\ell\cdot dC_p(q).
\end{equation}

For generic $p$, the transverse derivative satisfies $A^2\ne0$.
Indeed write $l(t)=L(1,t)$ in a projective chart. If $l,l'$ are
generically dependent, the projective direction of $l$ is constant.
If they are independent and $(l')^2$ vanishes identically, then they
span a totally isotropic plane for the bilinear product in $\C^4$.
This plane equals its bilinear orthogonal complement. Differentiating
$l\cdot l'=(l')^2=0$ gives $l''\cdot l=l''\cdot l'=0$.
Thus $l''\in\operatorname{span}(l,l')$ and
$(l\wedge l')'$ is a multiple of $l\wedge l'$. The plane is
constant on a dense open set. Both alternatives contradict the last
assertion of Lemma~\ref{lem:shell-leading}. A transverse derivative
differs from a nonzero multiple of $l'$ by a multiple of $l$ by
homogeneity, so the claim follows.

The quotient $\ell^{\perp_\cdot}/\C\ell$ is two-dimensional with
nondegenerate bilinear form; here $\perp_\cdot$ denotes bilinear
orthogonality. Choose a complex orthonormal basis with its first
vector parallel to the class of $A$. By \eqref{eq:shell-quotient}
the coordinates $x,y$ of the class of $Y$ are rational in $\sigma$,
$x$ is affine with nonzero slope, and $x^2+y^2=\kappa$ is constant.
If $\kappa\ne0$, the polynomial $\kappa-x^2$ has two simple zeros
and cannot be the square of a rational function. Hence $\kappa=0$,
and $y=ix$ or $y=-ix$ identically. The inverse image of the resulting
isotropic line in $\ell^{\perp_\cdot}$ is a totally isotropic plane.
It contains $Y(\sigma)-Y(0)=X(\sigma)-X(0)$ and $X'(0)$. Conjugating
and varying $v$ therefore gives, for $N=P-F_0$,
\begin{equation}\label{eq:shell-null-recentered}
                N(z)^2=N(z)\cdot Dz=(Dz)^2=0 .
\end{equation}
These identities extend from generic directions by holomorphicity.

The coefficient of $\tau$ in $N(\tau p)^2=0$ gives
$\ell\cdot(C(p)-F_0)=0$. At $\sigma=0$ the first identity of
\eqref{eq:shell-quotient} gives $\ell\cdot(C(p)+F_0)=0$, since
$F_0$ is real. Hence $L(p)\cdot F_0=0$ for generic $p$.
If $F_0\ne0$, its reality makes $F_0^{\perp_\cdot}$ a nondegenerate
three-dimensional space. Both $\ell$ and $L^*(v)$ lie in it and
are nonzero null vectors. The first and third identities of
\eqref{eq:shell-quotient}, divided by $\sigma$ as $\sigma\to\infty$,
give $\ell\cdot L^*(v)=0$ and $A\cdot L^*(v)=2\beta$.
Two mutually orthogonal null vectors in a nondegenerate three-space
are proportional, since a two-dimensional totally isotropic subspace
cannot fit in its one-dimensional orthogonal complement. Thus
$A\cdot L^*(v)=0$ by $A\cdot\ell=0$, a contradiction. This proves
$F_0=0$ and the assertion follows from \eqref{eq:shell-null-recentered}.
\end{proof}

\begin{proposition}\label{prop:shell-two}
The strict inequality $r_*/R_*<r/R$ is impossible for $n=2$.
\end{proposition}

\begin{proof}
By Lemma~\ref{lem:shell-null}, $V=\operatorname{im}D$ is a totally
isotropic plane in $\C^4$. Arrange a Hermitian orthonormal basis of
$V$ as the columns of a matrix $\mathbf e$. The conjugate columns form
an orthonormal basis of its Hermitian orthogonal complement. Write
\[
 P=\mathbf e H+\overline{\mathbf e}G,\qquad Dz=\mathbf e Mz,
\]
where $H,G$ are rational holomorphic maps to $\C^2$ and $M$ is an
invertible complex matrix. The bilinear identities give
$G\cdot Mz=0$ and $H\cdot G=0$. If $G=0$ identically, all values
of $L$ lie in $V$, contrary to Lemma~\ref{lem:shell-leading}.
With $J_0(x_1,x_2)=(-x_2,x_1)$, it follows that
\[
                     G=g J_0Mz,\qquad H=hMz .
\]
Indeed the first proportionality holds off $z=0$ and the second
holds first where $G\ne0$, then everywhere by the determinant
identity. Local division by a nonvanishing component shows that the scalar
factors are rational and holomorphic off $0$. They extend
across $0$ by Hartogs' theorem. Comparison with $P'(0)=D$ gives
$g(0)=0$, $h(0)=1$. Since $P^2=0$,
\begin{equation}\label{eq:shell-factor-norm}
        2|\widetilde F(z)|^2=|P(z)|^2
                      =|Mz|^2\big(|h(z)|^2+|g(z)|^2\big).
\end{equation}

We claim $M^\dagger M=\mu I$ for some $\mu>0$. Otherwise make a
unitary change of source coordinates so that its eigenvalues are
$\lambda_1\ne\lambda_2$. Complexification of \eqref{eq:shell-factor-norm}
on the inner sphere, as in \eqref{eq:shell-segre}, gives
\[
 2r_*^2=(\lambda_1z_1w_1+\lambda_2z_2w_2)
                 \big(h(z)h^*(w)+g(z)g^*(w)\big),\qquad z\cdot w=r^2 .
\]
Solve $c_1+c_2=r^2$ and $\lambda_1c_1+\lambda_2c_2=0$. Both
$c_i$ are nonzero. For generic $z$ with $z_1z_2\ne0$, substitution
of $w_i=c_i/z_i$ makes the right side zero. All factors are finite
there, since the substitution is an automorphism of $(\C^\times)^2$.
This contradicts $r_*>0$ and proves the claim.

The holomorphic pair $k=(h,g)$ now has constant squared norm
$\gamma_r=2r_*^2/(\mu r^2)$ on $S_r$, and its squared norm tends to
$\gamma_R=2R_*^2/(\mu R^2)$ as $|z|\to R$. Strict inequality gives
$\gamma_R>\gamma_r$, so $k$ is nonconstant. By the maximum principle
and continuity on $\overline B_r$, it is a proper holomorphic map
from $B_r$ to the ball of radius $\sqrt{\gamma_r}$ in $\C^2$.
Alexander's theorem \cite{Alexander1977}, after rescaling, makes
it a ball automorphism in unit ball coordinates. In particular
$k(p)=0$ at some $p\in B_r$. Since $k(0)=(1,0)$, we have $p\ne0$.
But the factorization of $P$ gives $\widetilde F(p)=0=\widetilde F(0)$,
contradicting the injectivity in Corollary~\ref{cor:round-shell-filling}.
\end{proof}

\begin{proof}[Conclusion of the proof of Theorem~\ref{thm:shell-rationality}]
The coefficient-rank argument following \eqref{eq:shell-two-factors}
and Propositions~\ref{prop:shell-three} and~\ref{prop:shell-two}
exclude strict inequality for every $n\ge2$. The equality statement
of Corollary~\ref{cor:round-shell-filling} therefore gives the
claimed similarity. Conversely, if the radius ratios agree, each
map $z\mapsto (R_*/R)Uz$, $U\in O(2n)$, is real linear and hence
pluriharmonic, and maps the two shells homeomorphically onto one another.
\end{proof}

\section{Harmonic quasiconformal maps onto Dini-smooth domains}
\label{sec:boundary}

We prove the boundary estimates for harmonic quasiconformal maps
used in Corollary~\ref{cor:bilipschitz-intro}. The direct rescaling
argument in Proposition~\ref{thm:qc-direct} supplies interior
nondegeneracy in the pluriharmonic case.

Throughout this section $d\ge2$ is the real dimension.
A regularized boundary distance is a positive smooth function
comparable to the boundary distance in a collar.
For any $K$-quasiconformal $C^1$ mapping $f$, let $\kappa\ge1$ be a
constant depending only on $d$ and $K$ that bounds its linear
distortion. With the notation of the introduction, we have
\begin{equation}\label{eq:linear-distortion}
 \|Df\|\le\kappa\ell(Df).
\end{equation}
This inequality holds everywhere:
it holds almost everywhere by quasiconformality and extends by continuity.
It also holds at a zero differential. Constants below may depend on the map,
the distortion, and the domains; no uniform normalization-independent
bi-Lipschitz constant is asserted.

\subsection{Boundary-distance estimates}

\begin{lemma}\label{lem:regularized-distance}
Let $\Omega\subset\R^d$ be a bounded $C^1$-Dini domain. There are a positive
regularized-distance function $\rho\in C^\infty(\Omega)$, a Dini modulus
$\omega$, and positive constants $c,C,\varepsilon$ such that, for every
$y\in\Omega$ with $0<\delta_\Omega(y)<\varepsilon$,
\begin{equation}\label{eq:regularized-distance}
 c\delta_\Omega(y)\le\rho(y)\le C\delta_\Omega(y),\quad
 c\le|\nabla\rho(y)|\le C,\quad
 \|D^2\rho(y)\|\le C\frac{\omega(\rho(y))}{\rho(y)}.
\end{equation}
\end{lemma}

These are Lieberman's regularized-distance estimates
\cite[Lemma~1.1 and Theorems~1.3 and~2.3]{Lieberman1985}.
The positive lower bound for $|\nabla\rho|$ on a smaller collar
follows from distance comparability and $C^1$ regularity at the
boundary.

\begin{lemma}\label{lem:distance-comparison}
Let $f:U\to\Omega$ be a harmonic quasiconformal homeomorphism between
bounded $C^1$-Dini domains in $\R^d$. Then there are $c,C>0$ such that
\begin{equation}\label{eq:distance-comparison}
 c\delta_U(x)\le\delta_\Omega(f(x))\le C\delta_U(x),
 \qquad x\in U.
\end{equation}
\end{lemma}

\begin{proof}
Take regularized distances $\rho_U$ and $\rho_\Omega$ from
Lemma~\ref{lem:regularized-distance}. Properness implies that
$\delta_\Omega(f(x))\to0$ as $x$ approaches $\partial U$.
Choose $h>0$ so small that $A=\{0<\rho_U<h\}$ lies in the source
collar and $f(A)$ lies in the target collar of the two distance
functions. Here the inner boundary $\{\rho_U=h\}$ is compact in
$U$. With $v=\rho_\Omega\circ f$, harmonicity gives
\[
 \Delta v=\sum_{\alpha,\beta=1}^d
    (\rho_\Omega)_{\alpha\beta}(f)
                   \langle\nabla f^\alpha,\nabla f^\beta\rangle .
\]
Also $|\nabla v|=|Df^T\nabla\rho_\Omega(f)|\ge c\ell(Df)$.
Together with \eqref{eq:linear-distortion} and the Hessian
estimate on $\rho_\Omega$ this gives
\begin{equation}\label{eq:collar-differential}
 |\Delta v|\le A_\Omega\frac{\omega_\Omega(v)}v|\nabla v|^2
\end{equation}
for a constant $A_\Omega>0$. The source distance estimate
and its gradient bound likewise give
\[
 |\Delta\rho_U|\le A_U
    \frac{\omega_U(\rho_U)}{\rho_U}|\nabla\rho_U|^2 .
\]
For either Dini modulus $\omega$ and its constant $A_0$, define
\[
 I(t)=\int_0^t\frac{\omega(s)}s\,ds,\qquad
 \Phi_\pm(t)=\int_0^t \exp\bigl(\pm A_0 I(s)\bigr)\,ds .
\]
These functions are comparable to $t$ near $0$. Differentiating
shows that $\Phi_-(v)$ is superharmonic and $\Phi_+(v)$ is
subharmonic for the target choice. Denote the source choices by
$\Psi_\pm$. The functions $\Psi_-(\rho_U)$ and $\Psi_+(\rho_U)$
are superharmonic and subharmonic, respectively.

All four functions tend to zero at the outer boundary of $A$.
On its compact inner boundary choose positive constants so that
$c\Psi_+(\rho_U)\le\Phi_-(v)$ and
$\Phi_+(v)\le C\Psi_-(\rho_U)$. The maximum principle gives
both inequalities on $A$. Comparability of the transformed and
original distances now gives \eqref{eq:distance-comparison} on
$A$. The remaining points lie in a compact subset of $U$,
where the inequality follows by changing the constants.
\end{proof}

\begin{lemma}\label{lem:upper-lipschitz}
Under the hypotheses of Lemma~\ref{lem:distance-comparison}, the mapping
$f$ is Lipschitz on $U$.
\end{lemma}

\begin{proof}
The estimate of Mateljevi\'c and Vuorinen
\cite[Theorem~2.8 and Corollary~2.9]{MateljevicVuorinen2010} gives
$\delta_U(x)\|Df(x)\|\le C\delta_\Omega(f(x))$.
By \eqref{eq:distance-comparison}, $\|Df\|$ is bounded.
A bounded connected $C^1$ domain is quasiconvex
\cite{MartioSarvas1979}, so integrating $Df$ along paths of length
bounded by a constant times the distance between their endpoints
gives the Lipschitz bound.
\end{proof}

\subsection{Bi-Lipschitz continuity}

Rescaling near a boundary point leads to a half-space. The following
lemma controls the normal component of the limiting harmonic map.

\begin{lemma}\label{lem:halfspace}
Let $H=\{x=(x_1,\ldots,x_d)\in\R^d:x_d>0\}$ be the upper
half-space. If $u$ is harmonic in $H$ and
$0\le u(x)\le Cx_d$, then there is a constant $a\ge0$ such that
$u(x)=a x_d$.
If $u(e_d)>0$, then $a>0$.
\end{lemma}

The bound on $u$ gives continuous zero boundary values.
If $u$ is not identically zero, it is positive by the strong minimum
principle, and the result is \cite[Theorem~7.22]{AxlerBourdonRamey2001}.

\begin{theorem}\label{thm:bilipschitz-criterion}
Let $f:U\to\Omega$ be a harmonic quasiconformal homeomorphism between
bounded $C^1$-Dini domains in $\R^d$. If $f$ is a local
$C^1$ diffeomorphism in $U$, then it is bi-Lipschitz.
\end{theorem}

\begin{proof}
Lemma~\ref{lem:upper-lipschitz} gives a global Lipschitz constant $L$.
We first prove $\inf_U\ell(Df)>0$. Otherwise choose
$x_k\in U$ with $\ell(Df(x_k))\to0$. Local nondegeneracy and
compactness force $r_k=\delta_U(x_k)\to0$. Choose a nearest
boundary point $q_k\in\partial U$. For large $k$,
$x_k-q_k=r_k\nu_U(q_k)$, where $\nu_U$ is the inward normal.
Take an orthogonal map $U_k$ with $U_ke_d=\nu_U(q_k)$ and put
\[
 D_k=\{\zeta:q_k+r_kU_k\zeta\in U\}.
\]
Uniform $C^1$ graph charts imply that $D_k$ converges locally
to the upper half-space $H$ in Lemma~\ref{lem:halfspace}.
The base point $e_d$ corresponds to $x_k$.
Let $p_k=f(x_k)$ and choose a nearest boundary point $a_k\in\partial\Omega$.
For $a\in\partial\Omega$, denote the inward unit normal by $\nu(a)$.
For large $k$, the nearest-point vector is an inward normal vector:
\[
 p_k-a_k=\delta_\Omega(p_k)\nu(a_k).
\]
Choose orthogonal maps $R_k$ sending $\nu(a_k)$ to $e_d$ and put
\[
 \Omega_k=r_k^{-1}R_k(\Omega-a_k),\qquad
 f_k(\zeta)=r_k^{-1}R_k\bigl(f(q_k+r_kU_k\zeta)-a_k\bigr).
\]
The $C^1$ boundary charts, uniformly on the compact boundary, show that
$\Omega_k$ converges locally to $H$ after taking a subsequence.
Moreover,
\[
 f_k(e_d)=\frac{\delta_\Omega(p_k)}{r_k}e_d,
\]
and these base values stay in a compact subinterval of the positive
$e_d$-axis by \eqref{eq:distance-comparison}.

The maps $f_k$ are harmonic and $L$-Lipschitz on $D_k$.
Their bounded values at $e_d$ give locally uniform bounds on $H$. Harmonic interior estimates and a diagonal
subsequence give $C^1_{\mathrm{loc}}$ convergence to a harmonic map
$f_\infty:H\to\overline H$.
Rescaling \eqref{eq:distance-comparison} gives
\[
 c\delta_{D_k}(\zeta)\le
 \delta_{\Omega_k}(f_k(\zeta))\le C\delta_{D_k}(\zeta).
\]
The locally convergent $C^1$ boundary graphs imply convergence of the
corresponding distance functions on bounded sets of interior points.
Thus, writing $u=(f_\infty)_d$ for the last real coordinate of
the limiting map,
\[
 c\zeta_d\le u(\zeta)\le C\zeta_d,\qquad \zeta\in H.
\]
Lemma~\ref{lem:halfspace} gives $u(\zeta)=a\zeta_d$ with $a>0$.
In particular $\|Df_\infty(e_d)\|\ge a$.

On the other hand, similarities preserve singular values up to their common
scale, and the source and target scales above are equal. Hence
\[
 \|Df_k(e_d)\|=\|Df(x_k)\|
 \le\kappa\ell(Df(x_k))\longrightarrow0.
\]
The $C^1$ convergence gives $Df_\infty(e_d)=0$, a contradiction.
The contradiction uses the limiting distortion inequality, without
requiring injectivity of the limit.

Choose $c_0>0$ with $\ell(Df)\ge c_0$. Then
$\|D(f^{-1})\|\le c_0^{-1}$ on $\Omega$.
A bounded connected $C^1$ domain is quasiconvex; let
$M_\Omega$ be a quasiconvexity constant for $\Omega$;
see \cite{MartioSarvas1979}. Integrating the differential of $f^{-1}$ along
such a curve gives
\[
 |f^{-1}(p)-f^{-1}(q)|\le c_0^{-1}M_\Omega|p-q|.
\]
Together with the upper bound, this proves bi-Lipschitz continuity.
\end{proof}

\begin{corollary}\label{cor:small-distortion-dini}
Let $d\ge2$, and let $f:U\to\Omega$ be a harmonic quasiconformal
homeomorphism between bounded $C^1$-Dini domains in $\R^d$. If
its outer dilatation satisfies $K_O(f)<3^{d-1}$, then $f$ is bi-Lipschitz.
\end{corollary}

\begin{proof}
After a target reflection if necessary, the theorem of Bo\v zin
and Mateljevi\'c \cite[Theorem~3.2]{BozinMateljevic2015} gives a
nonvanishing Jacobian. The inverse function theorem and
Theorem~\ref{thm:bilipschitz-criterion} give the conclusion.
\end{proof}

\begin{corollary}\label{thm:planar-dini}
Every harmonic quasiconformal homeomorphism between bounded planar
$C^1$-Dini domains is bi-Lipschitz.
\end{corollary}

\begin{proof}
Lewy's theorem \cite{Lewy1936} gives the local diffeomorphism
hypothesis of Theorem~\ref{thm:bilipschitz-criterion}.
\end{proof}

\begin{remark}
\label{rem:lyapunov-dini}
Every $C^{1,\alpha}$ boundary, $0<\alpha<1$, is $C^1$-Dini, since
$\int_0^{r_0} t^\alpha\,dt/t<\infty$. The converse fails: define $\omega(0)=0$ and
\[
 \omega(t)=\frac{1}{\log^2(e/t)},\qquad 0<t<e^{-2}.
\]
This is a Dini modulus near zero but is not $O(t^\alpha)$ for any $\alpha>0$.
Indeed, for any $0<r_0<e^{-2}$, its Dini integral is
$1/\log(e/r_0)$, whereas
$\omega(t)/t^\alpha\to\infty$ as $t\downarrow0$.
Near zero this modulus is increasing and concave. A graph with
derivative $g(t)=\operatorname{sgn}(t)\omega(|t|)$ near zero is
$C^1$-Dini but has no H\"older continuous derivative of positive
exponent there: $g$ has modulus at most $2\omega$, and
$|g(t)-g(0)|=\omega(|t|)$. Taking the integral of $g$ as a graph
perturbation and multiplying it by a smooth cutoff gives this local
boundary on a bounded planar domain. Thus the passage from Lyapunov
to $C^1$-Dini boundaries is a strict relaxation of the boundary assumption in the co-Lipschitz theorem of
Bo\v zin and Mateljevi\'c \cite{BozinMateljevic2020}.
Their theorem already allowed nonconvex Jordan targets; the improvement
relative to that theorem is the Dini boundary assumption and the finitely
connected setting. The removal of convexity at Dini regularity is the
improvement relative to \cite[Theorem~1.2]{Kalaj2015Dini}.
\end{remark}

\section*{AI declaration}
ChatGPT (OpenAI) suggested the proof idea connecting the local
phase-hypersurface problem with Milnor-fiber topology and monodromy,
and assisted in developing the low-dimensional shell rigidity arguments
in Appendix~\ref{app:round-shells}, checking mathematical arguments,
revising the exposition, and identifying relevant literature. The authors bear responsibility for the
mathematical arguments, references, and final text.

\section*{Acknowledgements and funding}
D. Kalaj is partially supported by the Ministry of Education, Science and
Innovation of Montenegro through the grants \emph{Mathematical Analysis,
Optimization and Machine Learning} and \emph{Complex-analytic and geometric
techniques for non-Euclidean machine learning: theory and applications}.
D. Zhong is partially supported by the Guangdong Basic and Applied Basic
Research Foundation under Grant nos.~2022A1515110967 and 2023A1515011809.


\begin{thebibliography}{99}

\bibitem{ACampo1973}
N.~A'Campo,
\emph{Le nombre de Lefschetz d'une monodromie},
Nederl. Akad. Wetensch. Proc. Ser. A \textbf{76} = Indag. Math. \textbf{35}
(1973), 113--118.

\bibitem{AdamowiczEtAl2020}
T.~Adamowicz, J.~J\"a\"askel\"ainen, and A.~Koski,
\emph{The Rad\'o--Kneser--Choquet theorem for $p$-harmonic mappings
between Riemannian surfaces}, Rev. Mat. Iberoam. \textbf{36}
(2020), 1779--1834.

\bibitem{AlessandriniNesi2009}
G.~Alessandrini and V.~Nesi,
\emph{Invertible harmonic mappings, beyond Kneser},
Ann. Sc. Norm. Super. Pisa Cl. Sci. (5) \textbf{8} (2009), 451--468.

\bibitem{Alexander1977}
H.~Alexander,
\emph{Proper holomorphic mappings in $\mathbb C^n$},
Indiana Univ. Math. J. \textbf{26} (1977), 137--146.

\bibitem{AlHelalLeblNandiAnnuli}
A.~Al Helal, J.~Lebl, and A.~K.~Nandi,
\emph{Proper maps of annuli}, to appear in Nagoya Math. J.,
arXiv:2511.08834v3.

\bibitem{AndreottiFrankel1959}
A.~Andreotti and T.~Frankel,
\emph{The Lefschetz theorem on hyperplane sections},
Ann. of Math. (2) \textbf{69} (1959), 713--717.

\bibitem{AstalaManojlovic2015}
K. Astala and V. Manojlovi\'c,
\newblock On Pavlovi\'c's theorem in space,
\newblock \emph{Potential Analysis} \textbf{43} (2015), no. 3, 361--370.
\newblock DOI: \texttt{10.1007/s11118-015-9475-4}.

\bibitem{AxlerBourdonRamey2001}
S.~Axler, P.~Bourdon, and W.~Ramey, \emph{Harmonic Function Theory},
2nd ed., Graduate Texts in Mathematics, vol.~137, Springer, New York,
2001.

\bibitem{BialynickiBirula1964}
A.~Bia\l ynicki-Birula, \emph{On Vietoris mapping theorem and its inverse},
Fund. Math. \textbf{53} (1964), 135--145.
DOI: \texttt{10.4064/fm-53-2-135-145}.

\bibitem{BojarskiIwaniec1983}
B.~Bojarski and T.~Iwaniec,
\emph{Analytical foundations of the theory of quasiconformal mappings in $\R^n$},
Ann. Acad. Sci. Fenn. Ser. A I Math. \textbf{8} (1983), 257--324.

\bibitem{BorelHaefliger1961}
A.~Borel and A.~Haefliger,
\emph{La classe d'homologie fondamentale d'un espace analytique},
Bull. Soc. Math. France \textbf{89} (1961), 461--513.
DOI: \texttt{10.24033/bsmf.1571}.

\bibitem{BozinMateljevic2015}
V. Bo\v zin and M. Mateljevi\'c,
\newblock Bounds for Jacobian of harmonic injective mappings in
\(n\)-dimensional space,
\newblock \emph{Filomat} \textbf{29} (2015), no. 9, 2119--2124.
\newblock DOI: \texttt{10.2298/FIL1509119B}.

\bibitem{BozinMateljevic2020}
V. Bo\v zin and M. Mateljevi\'c,
\newblock Quasiconformal and HQC mappings between Lyapunov Jordan domains,
\newblock \emph{Annali della Scuola Normale Superiore di Pisa, Classe di
Scienze} (5) \textbf{21} (2020), 107--132.
\newblock DOI: \texttt{10.2422/2036-2145.201708\_013}.

\bibitem{Bredon1993}
G.~E.~Bredon,
\emph{Topology and Geometry},
Graduate Texts in Mathematics, vol.~139, Springer-Verlag, New York, 1993.

\bibitem{Brown1962}
M.~Brown, \emph{Locally flat imbeddings of topological manifolds},
Ann. of Math. (2) \textbf{75} (1962), 331--341.

\bibitem{BshoutyHengartner1994}
D.~Bshouty and W.~Hengartner,
\emph{Univalent harmonic mappings in the plane},
Ann. Univ. Mariae Curie-Sk{\l}odowska Sect. A \textbf{48} (1994), 12--42.

\bibitem{BshoutyLyzzaik2010}
D.~Bshouty and A.~Lyzzaik,
\emph{Problems and conjectures in planar harmonic mappings},
J. Anal. \textbf{18} (2010), 69--81.

\bibitem{Choquet1945}
G.~Choquet, \emph{Sur un type de transformation analytique g\'en\'eralisant
la repr\'esentation conforme et d\'efinie au moyen de fonctions harmoniques},
Bull. Sci. Math. (2) \textbf{69} (1945), 156--165.

\bibitem{CisnerosSeadeSnoussi2009}
J.~L.~Cisneros-Molina, J.~Seade, and J.~Snoussi,
\emph{Refinements of Milnor's fibration theorem for complex singularities},
Adv. Math. \textbf{222} (2009), no.~3, 937--970.
DOI: \texttt{10.1016/j.aim.2009.05.010}.

\bibitem{DAngeloLebl2009}
J.~P.~D'Angelo and J.~Lebl,
\emph{On the complexity of proper holomorphic mappings between balls},
Complex Var. Elliptic Equ. \textbf{54} (2009), 187--204.
DOI: \texttt{10.1080/17476930902759403}.

\bibitem{DAngeloLeblPfister2012}
J.~P.~D'Angelo and J.~Lebl,
\emph{Pfister's theorem fails in the Hermitian case},
Proc. Amer. Math. Soc. \textbf{140} (2012), 1151--1157.
DOI: \texttt{10.1090/S0002-9939-2011-10841-4}.

\bibitem{Duren2004}
P.~Duren, \emph{Harmonic Mappings in the Plane},
Cambridge Tracts in Mathematics, vol.~156, Cambridge University Press,
Cambridge, 2004.

\bibitem{DurenHengartner1997}
P.~Duren and W.~Hengartner, \emph{Harmonic mappings of multiply connected
domains}, Pacific J. Math. \textbf{180} (1997), 201--220.

\bibitem{Forstneric1989}
F.~Forstneri\v c,
\emph{Extending proper holomorphic mappings of positive codimension},
Invent. Math. \textbf{95} (1989), 31--61.

\bibitem{GehringOsgood1979}
F. W. Gehring and B. G. Osgood,
\newblock Uniform domains and the quasi-hyperbolic metric,
\newblock \emph{Journal d'Analyse Math\'ematique} \textbf{36} (1979),
50--74.

\bibitem{GunningRossi1965}
R.~C.~Gunning and H.~Rossi,
\emph{Analytic Functions of Several Complex Variables},
Prentice-Hall, Englewood Cliffs, NJ, 1965.

\bibitem{HamadaKohr2015}
H.~Hamada and G.~Kohr,
\emph{Pluriharmonic mappings in $\C^n$ and complex Banach spaces},
J. Math. Anal. Appl. \textbf{426} (2015), 635--658.

\bibitem{Hatcher2002}
A.~Hatcher,
\emph{Algebraic Topology},
Cambridge University Press, Cambridge, 2002.

\bibitem{Hormander1990}
L. H\"ormander,
\newblock \emph{An Introduction to Complex Analysis in Several Variables},
\newblock third revised edition,
\newblock North-Holland, Amsterdam, 1990.

\bibitem{Iversen1986}
B.~Iversen, \emph{Cohomology of Sheaves}, Universitext, Springer-Verlag,
Berlin, 1986.

\bibitem{IwaniecKoskiOnninen2016}
T.~Iwaniec, A.~Koski, and J.~Onninen, \emph{Isotropic $p$-harmonic
systems in $2$D: Jacobian estimates and univalent solutions},
Rev. Mat. Iberoam. \textbf{32} (2016), 57--77.

\bibitem{IwaniecKovalevOnninen2011}
T.~Iwaniec, L.~V.~Kovalev, and J.~Onninen,
\emph{The Nitsche conjecture},
J. Amer. Math. Soc. \textbf{24} (2011), 345--373.

\bibitem{IwaniecOnninen2019}
T.~Iwaniec and J.~Onninen, \emph{Rad\'o--Kneser--Choquet theorem for
simply connected domains ($p$-harmonic setting)},
Trans. Amer. Math. Soc. \textbf{371} (2019), 2307--2341.

\bibitem{Kalaj2015Dini}
D. Kalaj,
\newblock Quasiconformal harmonic mappings between Dini-smooth Jordan domains,
\newblock \emph{Pacific Journal of Mathematics} \textbf{276} (2015), no. 1,
213--228.
\newblock DOI: \texttt{10.2140/pjm.2015.276.213}.

\bibitem{Kalaj2026Lewy}
D.~Kalaj,
\emph{A Lewy-type theorem for pluriharmonic mappings in $\C^2$
and bi-Lipschitz quasiconformal harmonic maps},
arXiv:2608.11948v2, 2026.
Preliminary version; principal results incorporated and extended here.

\bibitem{Kneser1926}
H.~Kneser, \emph{L\"osung der Aufgabe 41},
Jahresber. Deutsch. Math.-Verein. \textbf{35} (1926), 123--124.

\bibitem{KoskiOnninen2021}
A.~Koski and J.~Onninen,
\emph{Sobolev homeomorphic extensions},
J. Eur. Math. Soc. \textbf{23} (2021), no.~12, 4065--4089.

\bibitem{KrantzParks2002}
S.~G.~Krantz and H.~R.~Parks,
\emph{A Primer of Real Analytic Functions},
2nd ed., Birkh\"auser Advanced Texts, Birkh\"auser, Boston, 2002.

\bibitem{Laugesen1996}
R.~S.~Laugesen,
\emph{Injectivity can fail for higher-dimensional harmonic extensions},
Complex Variables Theory Appl. \textbf{28} (1996), 357--369.

\bibitem{Lewy1936}
H. Lewy,
\newblock On the non-vanishing of the Jacobian in certain one-to-one mappings,
\newblock \emph{Bulletin of the American Mathematical Society} \textbf{42} (1936), no. 10, 689--692.
\newblock DOI: \texttt{10.1090/S0002-9904-1936-06397-4}.

\bibitem{Lieberman1985}
G. M. Lieberman,
\newblock Regularized distance and its applications,
\newblock \emph{Pacific Journal of Mathematics} \textbf{117} (1985), no. 2,
329--352.

\bibitem{Lojasiewicz1964}
S.~\L ojasiewicz,
\emph{Triangulation of semi-analytic sets},
Ann. Scuola Norm. Sup. Pisa Cl. Sci. (3) \textbf{18} (1964), 449--474.

\bibitem{Lyzzaik1998}
A.~Lyzzaik, \emph{Univalence criteria for harmonic mappings in multiply
connected domains}, J. London Math. Soc. (2) \textbf{58} (1998), 163--171.

\bibitem{Manojlovic2009}
V.~Manojlovi\'c,
\emph{Bi-Lipschicity of quasiconformal harmonic mappings in the plane},
Filomat \textbf{23} (2009), no.~1, 85--89.

\bibitem{MartioSarvas1979}
O. Martio and J. Sarvas,
\newblock Injectivity theorems in plane and space,
\newblock \emph{Annales Academiae Scientiarum Fennicae. Series A I.
Mathematica} \textbf{4} (1979), 383--401.

\bibitem{MateljevicVuorinen2010}
M.~Mateljevi\'c and M.~Vuorinen,
\emph{On harmonic quasiconformal quasi-isometries},
J. Inequal. Appl. \textbf{2010} (2010), Article ID 178732, 19 pp.
DOI: \texttt{10.1155/2010/178732}.

\bibitem{Matousek2003}
J.~Matou\v{s}ek, \emph{Using the Borsuk--Ulam Theorem},
Universitext, Springer-Verlag, Berlin, 2003.

\bibitem{Milnor1963}
J.~Milnor, \emph{Morse Theory}, Annals of Mathematics Studies, no.~51,
Princeton University Press, Princeton, NJ, 1963.

\bibitem{Milnor1997}
J.~W.~Milnor,
\emph{Topology from the Differentiable Viewpoint},
Princeton Landmarks in Mathematics, Princeton University Press,
Princeton, NJ, 1997.

\bibitem{Naser1973}
M.~Naser,
\emph{A generalization of a theorem of Lewy on harmonic mappings},
Mat. Zametki \textbf{13} (1973), 189--194;
English translation in Math. Notes \textbf{13} (1973), 113--115.
DOI: \texttt{10.1007/BF01094227}.

\bibitem{Pavlovic2002}
M.~Pavlovi\'c,
\emph{Boundary correspondence under harmonic quasiconformal
homeomorphisms of the unit disk},
Ann. Acad. Sci. Fenn. Math. \textbf{27} (2002), 365--372.

\bibitem{Rado1926}
T.~Rad\'o, \emph{Aufgabe 41},
Jahresber. Deutsch. Math.-Verein. \textbf{35} (1926), 49.

\bibitem{Richberg1968}
R.~Richberg, \emph{Stetige streng pseudokonvexe Funktionen},
Math. Ann. \textbf{175} (1968), 257--286.

\bibitem{Schapira1995}
P.~Schapira, \emph{Tomography of constructible functions},
in Applied Algebra, Algebraic Algorithms and Error-Correcting Codes,
Lecture Notes in Computer Science, vol.~948, Springer, Berlin,
1995, 427--435.

\bibitem{Vaisala1971}
J. V\"ais\"al\"a,
\newblock \emph{Lectures on \(n\)-Dimensional Quasiconformal Mappings},
\newblock Lecture Notes in Mathematics, Vol. 229,
\newblock Springer-Verlag, Berlin--New York, 1971.

\bibitem{Wood1991}
J. C. Wood,
\newblock Lewy's theorem fails in higher dimensions,
\newblock \emph{Mathematica Scandinavica} \textbf{69} (1991), no. 2, 166--170.
\newblock DOI: \texttt{10.7146/math.scand.a-12375}.

\bibitem{ZhongKalaj2026MeanWidth}
D.~Zhong and D.~Kalaj,
\emph{Sharp mean-width and Jacobian bounds for Euclidean and
hyperbolic harmonic maps}, arXiv:2609.09316, 2026.

\bibitem{ZhongWang2026Lewy}
D.~Zhong and Z.-G.~Wang,
\emph{An all-dimensional Lewy theorem for pluriharmonic mappings
via Milnor monodromy}, arXiv:2609.31212v1, 2026.
Preliminary version; principal results incorporated and extended here.

\end{thebibliography}
\end{document}